\documentclass{article}

\usepackage[english]{babel}
\usepackage[utf8x]{inputenc}
\usepackage[T1]{fontenc}
\usepackage{amsmath,amssymb,amsthm,mathtools,mathdots,nicefrac,tikz}

\usepackage[a4paper,top=3cm,bottom=2cm,left=3cm,right=3cm,marginparwidth=1.75cm]{geometry}

\usepackage{comment}

\usepackage{amsmath}
\usepackage{amsthm}
\usepackage{placeins}

\usepackage{amssymb}
\usepackage{graphicx}
\graphicspath{{figures/}}
\usepackage{caption}
\usepackage[colorlinks=true, allcolors=blue]{hyperref}
\usepackage{enumerate}
\usepackage{algorithm}
\usepackage[noend]{algpseudocode}

\makeatletter
\def\BState{\State\hskip-\ALG@thistlm}
\makeatother

\usetikzlibrary{math,graphs,cd}
\usepackage{xcolor}

\usepackage{caption}
\usepackage{subcaption}

\theoremstyle{definition} 
\newtheorem{theorem}{Theorem}[section]

\newtheorem{corollary}[theorem]{Corollary}
\newtheorem{lemma}[theorem]{Lemma}
\newtheorem{proposition}[theorem]{Proposition}

\newtheorem{example}[theorem]{Example}

\newtheorem{question}[theorem]{Question}

\newcommand{\cT}{\mathcal{T}}
\newcommand{\cS}{\mathcal{S}}
\newcommand{\cX}{\mathcal{X}}

\DeclareMathOperator{\val}{val}
\DeclareMathOperator{\tw}{tw}
\DeclareMathOperator{\gon}{gon}

\DeclareMathOperator{\scw}{scw}
\DeclareMathOperator{\sn}{sn}

\DeclareMathOperator{\adh}{adh}

\DeclareMathOperator{\lw}{lw}

\DeclareMathOperator{\bw}{bw}

\DeclareMathOperator{\supp}{Supp}

\title{Scramble number and tree-cut decompositions}
\author{Lisa Cenek, Lizzie Ferguson, Eyobel Gebre, Cassandra Marcussen, Jason Meintjes,\\ Ralph Morrison, Liz Ostermeyer, Shefali Ramakrishna, and Ben Weber}
\date{}

\usepackage{graphicx}

\begin{document}

\maketitle

\begin{abstract}
    The scramble number of a graph is an invariant recently developed to study chip-firing games and divisorial gonality.  In this paper we introduce the screewidth of a graph, based on a variation of the existing literature on tree-cut decompositions.  We prove that this invariant serves as an upper bound on scramble number, though they are not always equal.  We study properties of screewidth, and present results and conjectures on its connection to divisorial gonality.
\end{abstract}

\section{Introduction}
The gonality of a graph is a combinatorial analog of gonality from divisor theory on algebraic curves. The \emph{(divisorial) gonality} of a graph $G$, denoted \(\gon(G)\), is the minimum degree of a rank $1$ divisor on $G$. 
Phrased in terms of chip-firing games, this is the minimum number of chips that can be placed on a graph so that they can be moved using chip-firing moves to eliminate one unit of debt, no matter where it is placed on the graph.
To upper-bound gonality by $k$, it suffices to provide a positive rank divisor on $G$ with degree $k$. Lower bounds, then, become the primary challenge in trying to compute gonality, as one must argue that every possible placement of fewer than $k$ chips on the graph fails to yield a winning strategy. 

Prior to 2020, the strongest known lower bound on gonality was the \emph{treewidth} of a graph, a well-studied graph invariant defined as the minimum width of a tree-decomposition. Seymour and Thomas \cite{st} proved a dual characterization of treewidth in terms of brambles, where a bramble is a collection of pairwise touching connected subgraphs; the treewidth of a graph is then one less than the maximum, over all brambles, of the minimum size of a hitting set for a bramble.  By \cite{debruyn2014treewidth}, we have \(\textrm{tw}(G)\leq \textrm{gon}(G)\).

In \cite{scramble}, the notion of a bramble was generalized to a \emph{scramble}, a collection of any connected subgraphs (not required to pairwise touch).  The order of a scramble is the minimum of two numbers:  the minimum size of a hitting set, and the minimum number of edges required to be deleted to put two of the scramble's subgraphs into separate components.  Letting the scramble number of a graph, denoted \(\textrm{sn}(G)\), be the maximum possible order of a scramble on that graph, we have \(\textrm{tw}(G)\leq \textrm{sn}(G)\leq \gon(G)\) \cite[Theorem 1.1]{scramble}.  Thus scramble number provides a strictly better lower bound on gonality than treewidth.

To compute the treewidth of a graph, one can give lower bounds by constructing brambles, and upper bounds by constructing tree decompositions.  Lower bounds on scramble number are given by specific scrambles, but there is no known tree decomposition structure to provide upper bounds on scramble number. This makes it difficult to compute scramble number, especially in the event that \(\textrm{sn}(G)<\gon(G)\).

In this paper, mirroring the max-min relationship between bramble number and treewidth, we present a tool to accompany scramble number: the \emph{screewidth} of a graph, denoted \(\scw(G)\). We prove the following  result.

\begin{theorem}\label{theorem:main}
For any graph G, we have \(\sn(G) \leq \scw(G)\).
\end{theorem}

Our paper is organized as follows.
In Section 2, we present definitions and background on general graph theory, on scramble number, and on screewidth. In Section 3 we prove several properties of the latter.  In Section 4 we prove Theorem \ref{theorem:main} and further study the relationship between scramble number and screewidth, including cases where \(\sn(G)<\scw(G)\).  In Section 5 we present results and possible directions for future work relating screewidth and gonality.

\medskip

\noindent \textbf{Acknowledgements.} The authors were supported by Williams College and the SMALL REU, and by the NSF via grants DMS-1659037 and DMS-2011743.  They thank Ben Baily for helpful discussions on screewidth.

\section{Background and definitions}

Throughout this paper, a graph $G = (V, E)$ is a connected, finite multigraph with vertex set \(V\) and edge set \(E\).  We allow multiple edges connecting a pair of vertices, but no loops connecting a vertex to itself.  If no two vertices are joined by multiple edges, we call \(G\) \emph{simple}.   The \emph{edge-connectivity} of a graph, denoted \(\lambda(G)\), is the smallest number of edges that can be deleted to disconnect \(G\).  Given two disjoint subsets \(A,B\subseteq V\), we let \(E(A,B)\subseteq E\) denote the set of edges with one endpoint in \(A\) and one endpoint in \(B\).  The \emph{valence} of a vertex \(v\) is the number of edges incident to \(v\); that is, \(\val(v)=\left|E(\{v\},\{v^C\}\right|\). A \(1\)-valent vertex is called a \emph{leaf}.

We say a graph \(H\) is a \emph{subgraph} of a graph \(G\) if it can be obtained by deleting edges and vertices from \(G\). Given a subset \(S\subset V(G)\), we let \(G[S]\) denote the subgraph induced by \(S\), with vertex set \(S\) and edge set equal to the subset of \(E\) consisting of edges with both endpoints in \(S\). We say that \(H\) is a \emph{subdivision} of \(G\) if it can be obtained by iteratively introducing \(2\)-valent vertices in the middles of edges.  We say that \(H\) is a \emph{refinement} of \(G\) if it can be obtained by ``smoothing'' \(2\)-valent vertices; more precisely, if \(v\) is a \(2\)-valent vertex with distinct neighbors \(u\) and \(w\), smoothing at \(v\) deletes \(v\) from the graph and adds an edge connecting \(u\) and \(w\).

We briefly recall several families of graphs that will be useful in our arguments and examples.  A \emph{tree} is a connected graph with no cycles. Three families of trees we will refer to are the \emph{path graphs} \(P_n\), consisting of \(n\) vertices connected in a path; the \emph{star graphs} \(K_{1,n}\), consisting of one vertex connected to \(n\) others, with no edges between the remaining \(n\) vertices; and the \emph{cubic caterpillar graphs} \(CCG_n\) for \(n\) even, consisting of a path of length \((n/2)+1\) with a leaf added to each internal vertex.  Examples of these graphs are illustrated in Figure \ref{figure:tree_examples}. The \emph{cycle graph on \(n\) vertices}, denoted \(C_n\), is the graph with \(n\) vertices arranged in a cycle, each vertex connected to its two neighbors.  The \emph{complete graph on \(n\) vertices}, denoted \(K_n\), is the simple graph with \(n\) vertices in which each vertex is connected to all others.

\begin{figure}[hbt]
    \centering
    \includegraphics{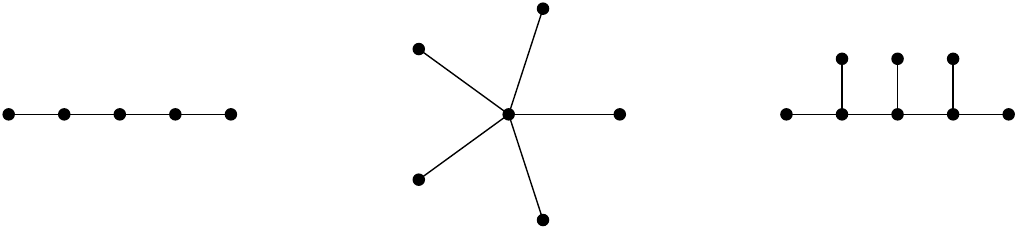}
    \caption{The path graph \(P_5\), the star graph \(K_{1,5}\), and the cubic caterpillar graph \(CCG_8\).}
    \label{figure:tree_examples}
\end{figure}

We close our recollection of graph theory with the definition of two types of graph products: Cartesian and rooted. Given graphs \(G\) and \(H\), their \emph{Cartesian product} \(G\square H\) is the graph with vertex set \(V(G)\times V(H)\) and an edge between \((u,v)\) and \((u',v')\) if and only if \(u=u'\) and \(vv'\in E(H)\), or \(v=v'\) and \(uu'\in E(H)\) (in both cases the number of edges included between two vertices matches the number coming from \(G\) or \(H\)).  A few famous examples of product graphs are the grid graphs \(G_{m,n}=P_m\square P_n\), the stacked prism graphs \(Y_{m,n}=C_m\square P_n\),  the toroidal grid graphs \(T_{m,n}=C_m\square C_n\), and the \(n\)-dimensional hypercube graph \(Q_n\), which is the Cartesian product of \(n\) compies of \(K_2\).

 Given two graphs \(G\) and \(H\), with a vertex \(v\in V(H)\), the \emph{rooted product of \(G\) with \(H\) at \(v\)}, denoted \(G\circ(H,v)\), consists of a copy of \(G\) with \(|V(G)|\) copies of \(H\), each glued at the vertex \(v\) to a different vertex of \(G\). If \(H\) is vertex-transitive, meaning every vertex is indistinguishable, we may omit the choice of \(v\) and write \(G\circ H\).

 We now recall definitions from \cite{scramble}, which introduced the definitions of scramble and scramble number.  A \emph{scramble} $\cS$ on a graph $G$ is a collection \(\{E_1,\ldots,E_\ell\}\) of subsets \(E_i\subset V(G)\), referred to as \emph{eggs}, such that \(G[E_i]\) is connected for all \(i\).  A \emph{hitting set} for $\cS$ is a subset $H \subseteq V(G)$ such that $H \cap E_i \neq \emptyset$ for all $E_i \in \cS$. The minimum size of a hitting set for $\cS$ is denoted $h(\cS)$. An \emph{egg-cut} for \(\cS\) is a set \(T\subset E(G)\) such that the graph \(G-T\) has at least two distinct connected components containing an egg.  The minimum size of an egg-cut for \(\cS\) is denoted \(e(\cS)\).  We remark that in a minimum egg-cut, the graph is split into exactly two components; otherwise there are edges that could have been excluded from \(T\) while still disconnecting two eggs.  We can thus write such egg-cuts as \((A,A^C)\), where \(A\) and \(A^C\) are the vertices of the connected components of \(G-T\), and \(T=E(A,A^C)\) is the collection of edges conencting \(A\) and \(A^C\).  We define the \emph{order of a scramble} to be the minimum of $h(\cS)$ and \(e(\cS)\):
 \[||\mathcal{S}||=\min\{h(\cS),e(\cS)\}.\]
Finally, the \emph{scramble number} of a graph \(G\) is the maximum order of any scramble on that graph:
\[\sn(G)=\max_{\cS\textrm{ on }G}||\mathcal{S}||.\]

\begin{figure}[hbt]
    \centering
    \includegraphics{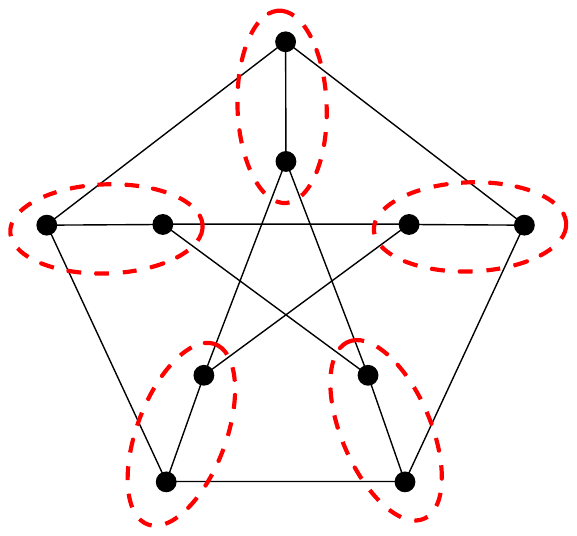}
    \caption{The Petersen graph with a scramble on it.}
    \label{figure:petersen_scramble}
\end{figure}

For an example of a scramble, let \(G\) be the Petersen graph illustrated in Figure \ref{figure:petersen_scramble}, with five collections of vertices circled; these  form the five eggs of a scramble \(\mathcal{S}\).  Since the five eggs happen to be disjoint, a minimal hitting set will consist of one vertex in each egg, so \(h(\mathcal{S})=5\).  Note that between any two eggs \(E_i\) and \(E_j\), there are four edge-disjoint paths connecting them:  two along the outer cycle, and two along the inner cycle.  To separate \(E_i\) and \(E_j\) into distinct components, at least one edge from each path must be deleted.  It follows that \(e(\cS)\geq 4\), and in fact \(e(\cS)= 4\); for instance, deleting the four edges incident to one of the eggs yields an egg-cut.  Thus \(||\cS||=\min(5,4)=4\), and we know that \(\sn(G)\geq 4\).  Proving \(\sn(G)= 4\) would take a lot more work, a priori:  one would need to argue that no scramble on the Petersen graph has order \(5\) or more.

We now set up the framework necessary to define screewidth. Let \(G\) be a graph and let $\cT = (T, \mathcal{X})$, where $T$ is a tree and for every vertex $b \in \ V(T)$, there exists a set $X_b \in \mathcal{X}$, referred to as a \emph{bag}, whose elements are vertices in $V(G)$; note that \(X_b\) may be the empty set. If
        
        \begin{itemize}
            \item $\displaystyle \bigcup_{b \in V(T)}X_b = V(G)$, and
            \item for any distinct $b, d \in V(T), X_b \cap X_d = \varnothing$,
        \end{itemize}
 then we call $\cT$ a \emph{tree-cut decomposition} of $G$.  To distinguish between the edges and vertices of $T$ and of $G$, we will call an element of $V(T)$ a \emph{node} and an element of $E(T)$ a \emph{link};  we will continue to refer to elements of $V(G)$ and $E(G)$ as vertices and edges, respectively.

Tree-cut decompositions have been previously studied in \cite{wollan,dualobjects}.  However, our definition of width will differ slightly from what appears in earlier literature.  Given a tree-cut decomposition $\cT = (T, \mathcal{X})$ of $G$, deleting any link $l \in E(T)$ will disconnect $T$ since \(T\) is a tree. The \emph{(link) adhesion} of $l$, denoted $\adh(l)$, is the set of all edges $e \in E(G)$ with endpoints in bags \(X_b\) and \(X_d\) such that the nodes \(b\) and \(d\) are in different components of \(T-l\). Similarly, deleting any non-leaf node $b \in V(T)$ will disconnect $T$. The \emph{(node) adhesion} of $b$, denoted $\adh(b)$, is the set of all edges \(e\in E(G)\) whose endpoints are in bags \(X_c\) and \(X_d\) with \(c\) and \(d\) in different components of $T-b$.

The \emph{link width} of $\cT=(T, \mathcal{X})$, denoted $\lw(\cT)$, is \[\lw(\cT)=\max_{l \in E(T)}\{|\adh(l)|\}.\]  The \emph{bag width} of  $\cT=(T, \mathcal{X})$, denoted $\bw(\cT)$, is \[\bw(\cT)=\max_{b \in V(T)}\{|X_b| + |\adh(b)|\}.\]  The \emph{width} of a tree-cut decomposition, denoted $w(\cT)$, is $\max\{\lw(\cT), \bw(\cT)\}$.  Finally, the \emph{screewidth}\footnote{The etymology of screewidth is two-fold. First, it comes from the word \emph{scree}, meaning a loose collection of rocks, of which the disjoint bags \(X_b\) are reminiscent.  Second, it is a portmanteau of \emph{scramble} and \emph{treewidth}.} of a graph $G$, denoted $\mathrm{scw}(G)$, is the minimum possible width of a tree-cut decomposition on \(T\). If \(w(\mathcal{T}=\scw(G)\), we call \(\mathcal{T}\) \emph{optimal}.

\begin{figure}[hbt]
    \centering
    \includegraphics[scale=0.7]{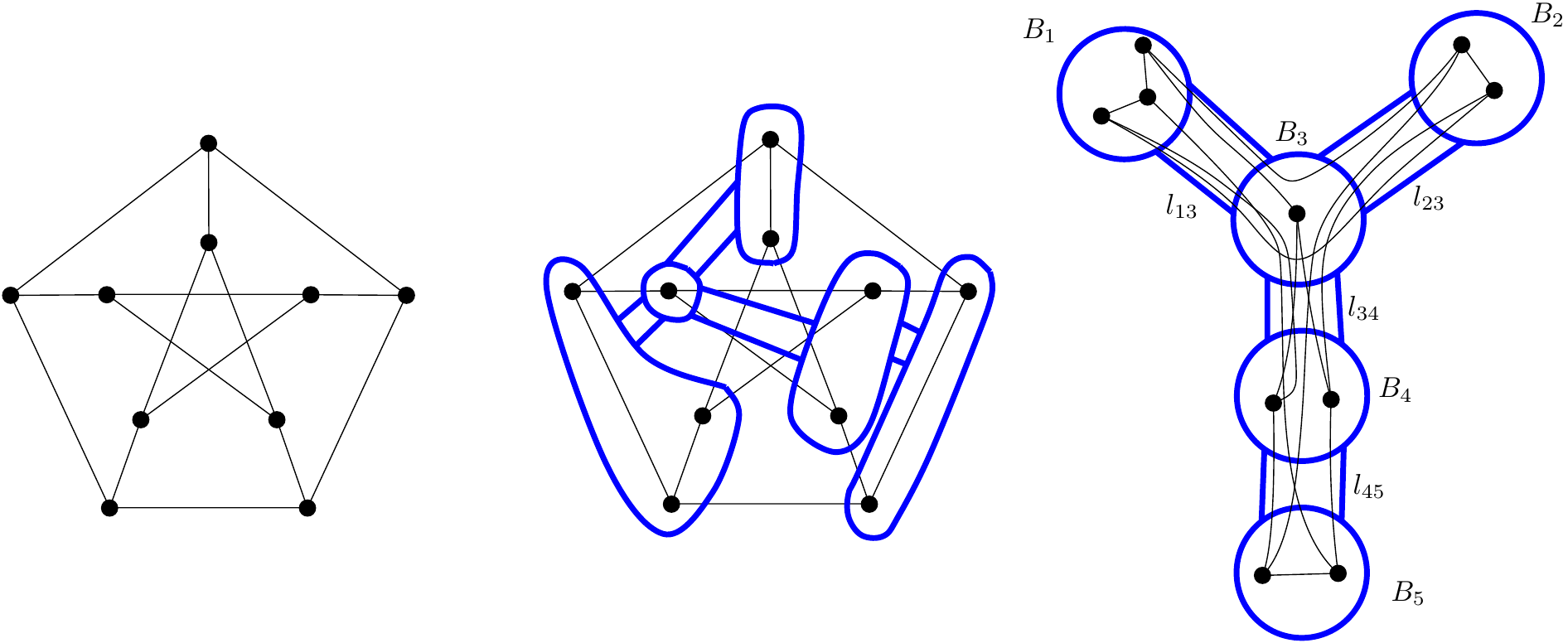}
    \caption{The Petersen graph, with two illustrations of the same tree-cut decomposition of width~\(7\).}
    \label{fig:petersen_example}
\end{figure}

As an example, consider the Petersen graph \(G\) illustrated on the left in Figure \ref{fig:petersen_example}.  In the middle is an example of tree-cut decomposition \(\mathcal{T}=(T,\mathcal{X})\).  We remark that bags corresponding to adjacent nodes of \(T\) need not have any vertices connected with one another in \(G\); nor do the vertices in a bag need to induce a connected subgraph of \(G\).

On the right in Figure \ref{fig:petersen_example} is a more illustrative way of drawing the same tree-cut decomposition, where the tree is drawn first, with all vertices of \(G\) in the nodes corresponding to their respective bags; and then all edges connecting vertices in distinct bags \(X_b\) and \(X_d\) are drawn to traverse the unique path from \(b\) to \(d\) in \(T\). Here we blur the distinction between a node \(b_i\in T\) and the corresponding bag \(B_i:=B_{b_i}\) in \(\mathcal{X}\).  The adhesion of a link can be quickly seen as the number of edges passing through that link, so \(|\adh(l_{13})|=5\), \(|\adh(l_{23})|=|\adh(l_{45})=4\)|, and \(|\adh(l_{34})|=6\); so, \(\lw(\mathcal{T})=6\).  The adhesion of a node can similarly be seen as the number of edges passing through the corresponding bag with neither endpoint in that bag, so \(\adh(b_3)=6\) and \(\adh(b_4)=2\); all other nodes correspond to leaves, and thus have adhesion equal to \(0\). Adding each node adhesion to the number of vertices in the corresponding bag, we find that \(\bw(T)=|B_3|+|\adh(B_3|)=1+6=7\).  Thus, \(w(\mathcal{T})=\max\{\lw(\mathcal{T}),\bw(\mathcal{T})\}=\max\{6,7\}=7\).

\begin{figure}
    \centering
    \includegraphics{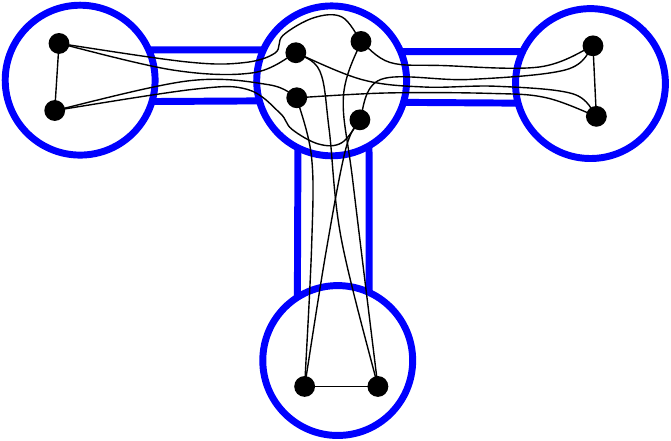}
    \caption{A tree-cut decomposition of the Petersen graph with width \(4\).}
    \label{fig:petersen_decomposition_optimal}
\end{figure}

This turns out not to be an optimal tree-cut decomposition of the Petersen graph.  As shown in Figure \ref{fig:petersen_decomposition_optimal}, the Petersen graph has a tree-cut decomposition of width \(4\), so we have \(\scw(G)\leq 4\).  Since we already saw that \(\sn(G)\geq 4\), our Theorem \ref{theorem:main} gives us \(4\leq \sn(G)\leq \scw(G)\leq 4\), so \(\sn(G)=\scw(G)=4\).

We close this section by briefly recalling the mechanics of chip-firing and the definition of gonality.  We only recall what is necessary for the content of Section \ref{section:gonality}; see \cite{sandpiles} for more details.

Given a graph \(G\), a \emph{divisor} on \(G\) is an assignment of integers to the vertices of \(G\).  These are intuitively thought of as poker chips placed on the vertices of the graph (with some poker chips representing debt, allowing for negative integers). If every integer in a divisor \(D\) is nonnegative, we say \(D\) is \emph{effective}, denoted \(D\geq 0\).  The sum of the integers assigned to the vertices by \(D\) is called the \emph{degree} of \(D\), denoted \(\deg(D)\).

For a vertex \(v\in V(G)\), the \emph{chip-firing move at \(v\)} transforms a divisor \(D\) into a new divisor \(D'\) as follows.  A total of \(\textrm{val}(v)\) chips are removed from \(v\) and redistributed to its neighbors, one along each edge incident to \(v\).  Note that the total number of chips is preserved, so \(\deg(D)=\deg(D')\).  If two divisors \(D\) and \(D'\) differ by a sequence of chip-firing moves, we say that \(D\) and \(D'\) are \emph{equivalent}, written \(D\sim D'\).

In addition to  firing vertices individually, one can perform a \emph{subset-firing move}, wherein every vertex in a set \(S\subset V(G)\) is fired simultaneously (or equivalently, in any order).  Consider two equivalent effective divisors \(D\) and \(D'\) and a sequence of chip-firing moves transforming \(D\) into \(D'\); since firing every vertex once has no effect on a divisor, we may assume not all vertices are fired in this sequence.  This sequence of firing-moves then admits a \emph{level-set decomposition}
\[\emptyset \subsetneq U_1\subseteq U_2\subseteq\ldots \subseteq U_k\subsetneq V(G),\]
where subset-firing \(U_1,U_2,\ldots,U_k\) transforms \(D\) into \(D'\), such that the resulting sequence of divisors
\[D=D_0,D_1,D_2,\ldots,D_{k-1},D_k=D'\] are all effective. See \cite{db} for more details on level-set decompositions.

We say a divisor \(D\) has \emph{positive rank} if for every \(v\in V(G)\), \(D\) is equivalent to an effective divisor \(D'\) with at least one chip on \(v\).  The \emph{(divisorial) gonality of a graph}, denoted \(\gon(G)\), is the minimum degree of a divisor of positive rank. Intuitively, this is the smallest number of chips that can be moved around by subset-firing moves to cover any vertex in the graph without ever introducing debt.

\section{Properties of screewidth}\label{section:properties_of_screewidth}

Throughout this section, we develop results on screewidth, saving its connections to scramble number for Section \ref{section:sn_and_scw}.  We will occasionally forward reference a result from that section; however, the later proof will not rely on any results from this section.

We note that every graph has a \emph{trivial} tree-cut decomposition consisting of a tree on one vertex, and every vertex of the underlying graph in the same bag. The width of the trivial tree-cut decomposition is always equal to $|V(G)|$.

\begin{proposition}\label{prop:scw_n-1}
For a graph $G$ on \(n\) vertices we have $\scw(G) \leq n$, with equality if and only if \(\lambda(G)\geq n\). In particular, if $G$ is simple, then $\scw(G) \leq n - 1$.
\end{proposition}
\begin{proof}
The trivial tree-cut decomposition proves $\scw(G) \leq |V(G)|$.  Now assume \(\lambda(G)\geq |V(G)|\), and that \(\mathcal{T}=(T,\mathcal{X})\) is a tree-cut decomposition of \(G\).  Either \(T\) has a single bag, in which case \(w(\mathcal{T})=\bw(\mathcal{T})=n\); or \(T\) has at least one link \(l\), whose adhesion forms an edge-cut for \(G\).  It follows that \(w(\mathcal{T})\geq \adh(l)\geq\lambda(G)\geq n\).  Now, suppose \(G\) satisfies \(\lambda(G)<|V(G)|\).  By deleting \(\lambda(G)\) edges, we can split \(G\) into two connected components, say with vertex sets \(A\) and \(B\).  Consider the tree-cut decomposition \(\mathcal{T}\) with \(T=K_2\) and \(\mathcal{X}=\{A,B\}\).  The width of this decomposition is the maximum of \(\lambda(G)\), \(|A|\), and \(|B|\), all of which are strictly smaller than \(n\).  Thus \(\scw(G)\leq n-1\).

 If $G$ is simple, then \(\lambda(G)\leq n-1\), giving the final claim.
\end{proof}

For simple graphs, we can achieve an even better upper bound on screewidth using the independence number of a graph. Recall that a set \(S\subseteq V(G)\) is called  \emph{independent} if no two vertices in \(S\) share an edge; and the \emph{independence number} \(\alpha(G)\) of a graph is the maximum size of an independent set on \(G\).

\begin{lemma}\label{lemma: scw_n-alpha}
For a simple graph $G$ on $n$ vertices, we have $\scw(G) \leq n - \alpha(G)$.
\end{lemma}
\begin{proof}
Let $S \subseteq V(G)$ be an independent set such that $|S| = \alpha(G)$. Construct a tree-cut decomposition as follows. Let $T = K_{1,\alpha(G)}$. Let the bag corresponding to the central node of the star be the set $V(G) - S$, and let each vertex of $S$ be in its own bag corresponding to a leaf node of $T$. Then the adhesion of any link in $\cT$ is at most $n-\alpha(G)$, since the degree of any vertex in $S$ is at most $n-\alpha(G)$. Furthermore, for \(b\) the central node we have \(\adh(b)=0\) since \(S\) is independent, yielding $|X_b|+|\adh({b})|=|V(G)-S|=n-\alpha(G)$; and for \(b\) a leaf node we have \(|X_b|+|\adh(b)|=1+0=1\). Thus, the width of this tree-cut decomposition equal to $n-\alpha(G)$. We conclude that $\scw(G) \leq n - \alpha(G)$.
\end{proof}

The following lemma is useful in considering node adhesions.

\begin{lemma}
Let $\mathcal{T} = (T, \mathcal{X})$ be a tree-cut decomposition on $G$. Suppose $vw \in E(G)$ with $v$ and $w$ in bags $X_{c}, X_{d}$ respectively, and assume $c$ and $d$ are nonadjacent in $T$. Then the edge $vw \in \adh{(b)}$ for every $b \in V(T)$ with $b$ on the path in $T$ from $c$ to $d$. 
\end{lemma}
\begin{proof}
Assume $b$ lies on the path in $T$ from $c$ to $d$. Thus, removing $b$ from $T$ leaves $c$ and $d$ in different components. So, the edge $vw \in E(G)$ is in the node adhesion of $b$.
\end{proof}

The following several results show that certain unexpected structures can be required for optimal tree-cut decompositions. In the next proposition we show that there exist graphs $G$ such that any optimal tree-cut decomposition $\cT=(T,\cX)$ has at least one empty bag.  Consider $K_3 \circ K_n^l$ where $K_n^l$ denotes a complete graph on $n$ vertices where the vertices are pairwise connected by $l$ parallel edges where $n < l$; see Figure \ref{fig:empty bags needed}.

\begin{figure}[hbt]
    \centering
    \includegraphics[scale=.5]{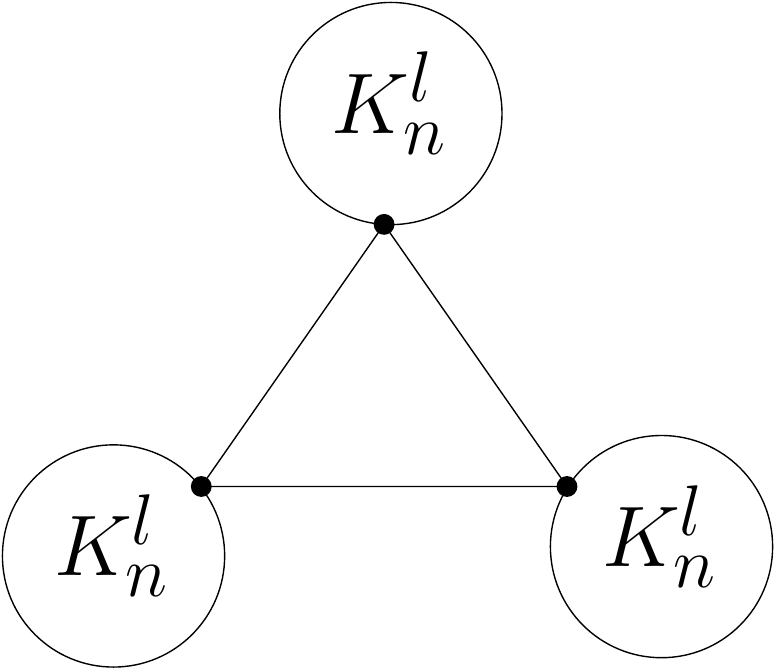}
    \caption{The rooted product $K_3 \circ K_n^l$.}
    \label{fig:empty bags needed}
\end{figure}

\begin{proposition}\label{prop:scw_empty_bag}
For \(n\geq 3\) we have \(\scw(K_3 \circ K_n^l)=n\), and that any tree decomposition with width \(n\) contains at least one empty bag.
\end{proposition}

\begin{proof}
 Figure \ref{fig:empty_bags_optimal} presents on the left a tree-cut decomposition of width \(n\), accomplished by allowing an empty bag.  Thus \(\scw(G)\leq n\).

\begin{figure}[hbt]
    \centering
    \includegraphics[scale=.5]{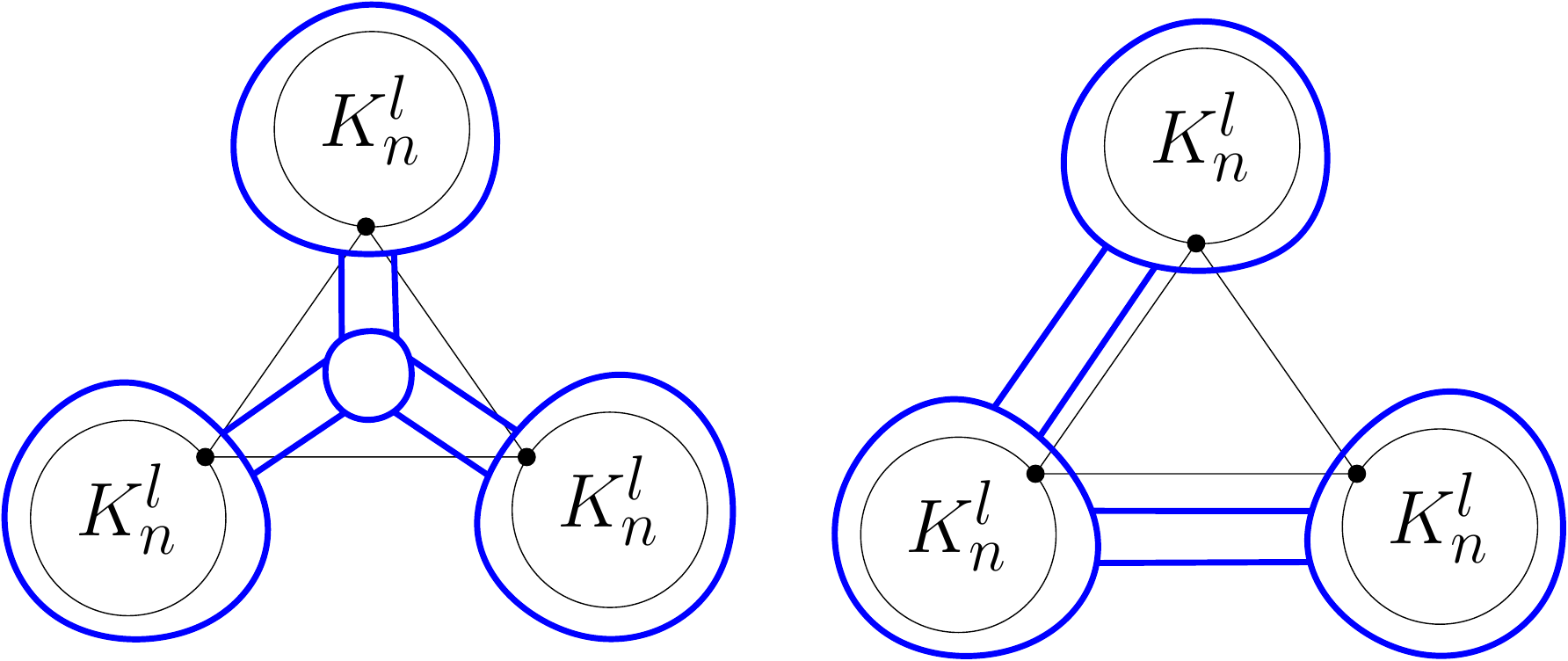}
    \caption{A decomposition with an empty bag on the left, which has lower width than the decomposition with no empty bags on the right.}
    \label{fig:empty_bags_optimal}
\end{figure}

Let \(\mathcal{T}\) be an optimal tree-cut decomposition for the graph.  First, we argue that the vertices in any copy of \(K_n^l\) must be in a common bag.  For if \(v,w\in V(G)\) are in the same \(K_n^l\) but in different bags \(X_b,X_d\), then any link on the path between \(b\) and \(d\) has adhesion at least \(l>n\), giving a higher width of the whole decomposition.  Thus in an optimal decomposition, some bag has at least \(n\) vertices, implying that \(\scw(G)\geq n\), and hence that \(\scw(G)=n\).  From here we know that there are at least three non-empty bags in an optimal tree-cut decomposition, each containing precisely the vertices from a copy of \(K_n^l\) (otherwise some bag would have more than \(n\) vertices, increasing the width).  If such a decomposition does not have any empty bags, tree \(T\) must be the unique tree with three vertices, namely the path \(P_3\); this decomposition is  illustrated on the right in Figure \ref{fig:empty_bags_optimal}.  But this decomposition has width \(n+1\), coming from the node \(b\) in the middle of the path having \(|X_b|+|\adh(b)|=n+1\).  Thus any tree-cut decomposition of \(G\) with width \(\scw(G)=n\) has an empty bag.
\end{proof}

Although empty bags are sometimes required, the following result shows we can assume that they have valence exactly \(3\).

\begin{proposition}\label{prop:empty_bags_trivalent}
Any graph \(G\) has a tree-cut decomposition \(\mathcal{T}\) with \(w(\mathcal{T})=\scw(G)\) such that any empty bag in  \(\mathcal{T}\) has valence \(3\).
\end{proposition}

\begin{proof} Let \(\mathcal{T}\) be a tree-cut decomposition of \(G\) with \(w(\mathcal{T})=\scw(G)\).
First we note that if \(X_b\) is an empty bag such that \(b\) has valence \(1\) (that is, an empty bag corresponding to a leaf), then \(|X_b|=0\) and \(\adh(b)=0\).  It follows we may iteratively delete any such \(b\) from the tree-cut decomposition \(\mathcal{T}\) until none remain without changing the width.

Next, suppose that \(X_c\) is an empty bag where \(c\) has valence \(2\), say with neighboring nodes \(b\) and \(d\) connected by links \(l_b\) and \(l_d\) respectively.  Note that \(|X_c|+|\adh(c)|=|\adh(c)|\), and that \(\adh(c)=\adh(l_b)\), since deleting \(b\) from the tree causes the exact same disconnections as deleting \(l_b\). The same argument shows \(\adh(c)=\adh(l_d)\).  Now, consider the new tree-cut decomposition \(\mathcal{T}'\) obtained by deleting \(c\) and connecting \(b\) and \(d\) with a new link \(l\).  We have \(\adh(l)=\adh(l_b)=\adh(l_d)=|X_c|+|\adh(c)|\), and no other bag sizes, vertex adhesions, or link adhesions have changed, so \(w(\mathcal{T}')=w(\mathcal{T})\).  Thus we may iteratively smooth over \(2\)-valent nodes corresponding to empty bags until none remain.

Finally, suppose that \(X_b\) is an empty bag where \(b\) has valence strictly great than \(3\), say with neighboring bags \(b_1,\ldots,b_k\) connected by links \(l_1,\ldots,l_k\).  Delete \(b\), and introduce two new nodes, \(b'\) and \(b''\), connected by a link \(l\). Connect \(b'\) to \(b_1\) and \(b_2\) by links \(l_1'\) and \(l_2'\), and connect \(b''\) to \(b_3,\ldots,b_k\) by links \(l_3',\ldots,l_4'\).  Letting \(X_{b'}=X_{b''}=\emptyset\), we have a new tree decomposition \(\mathcal{T}'=(T',\mathcal{X}')\).  We make the following observations, using subscripts to indicate which tree-cut decomposition is being considered:
\begin{itemize}
    \item Deleting \(b\) in \(T\) causes at least as many edge separations as deleting \(l\) in \(T'\), so \(\adh_{\mathcal{T}'}(l)\subseteq \adh_{\mathcal{T}}(b)\), and thus \(|\adh_{\mathcal{T}'}(l)|\leq |\adh_{\mathcal{T}}(b)|\).  Similarly, deleting \(b\) in \(T\) causes at least as many edge separations as deleting \(b'\) or \(b''\) in \(T'\), giving us \(|\adh_{\mathcal{T}'}(b')|\leq |\adh_{\mathcal{T}}(b)|\) and \(|\adh_{\mathcal{T}'}(b'')|\leq |\adh_{\mathcal{T}}(b)|\).
    \item For all \(i\), we have \(\adh_{\mathcal{T}'}(l_i')=\adh_{\mathcal{T}}(l_i)\), since deleting either \(l_i\) or \(l_i'\) results in the same partition of the vertices of \(G\).
    \item All unmodified links, nodes, and bags have the same adhesion and size as before.
\end{itemize}
It follows that in constructing our new tree-cut decomposition, we did not increase width, and we replaced a node corresponding to an empty bag of valence \(k\geq 4\) with two nodes, one of valence \(3\) and one of valence \(k-1\).  Thus we may repeat this procedure until no empty bags of valence more than \(4\) exist.
\end{proof}

This allows us to bound the required number of empty bags in an optimal tree-cut decomposition, relative to the number of non-empty bags.

\begin{corollary}\label{corollary:number_of_leaves} Any graph has an optimal tree-cut decomposition with all \(m\) leaf nodes corresponding to nonempty bags, with at most \(m-2\) nodes corresponding to empty bags.
\end{corollary} 

\begin{proof}
If a tree has \(k\) vertices, then it has \(k-1\) edges.   Letting \(k_i\) denote the number of vertices of valence \(i\), we have
\[k_1+2k_2+3k_3+\cdots=2k-2,\]
since each edge is shared by two vertices.  Noting that \(k=k_1+k_2+k_3+\cdots\), we have
\[-k_1+k_3+2k_4+\cdots=-2,\]
or \[k_3+2k_4+\cdots=k_1-2.\]
It follows that \(k_3\leq k_1-2\); that is, the number of vertices of valence \(3\) in a tree is at most \(2\) less than the number of leaves.

Let \(G\) be a graph, with \(\mathcal{T}\) an optimal tree-cut decomposition with all empty bags having valence \(3\). Then if \(m\) is the number of (necessarily nonempty) leaf nodes,  there are at most \(m-2\) empty bags, as claimed.
\end{proof}

We apply our new understanding of empty bags to prove that there exists a graph whose optimal tree-cut decomposition requires a disconnected bag.

\begin{proposition}
Let \(G\) be the graph pictured on the left in Figure \ref{fig:scw=2_disconnected}.  Any tree-cut decomposition \(\mathcal{T}\) of \(G\) with \(w(\mathcal{T})=\scw(G)\) has a nonempty bag \(X_b\) with \(G[X_b]\) disconnected.
\end{proposition}

\begin{figure}[hbt]
    \centering
    \includegraphics{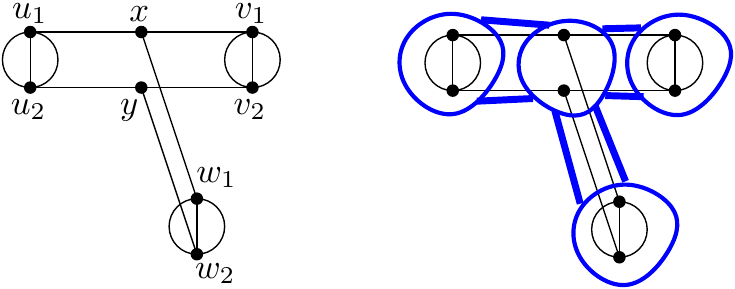}
    \caption{A graph of screewidth \(2\), requiring a disconnected bag.}
    \label{fig:scw=2_disconnected}
\end{figure}

\begin{proof}
First we note that \(\scw(G)\leq 2\) thanks to the tree-cut decomposition on the right in Figure \ref{fig:scw=2_disconnected}; and since \(G\) is not a tree, we have \(\scw(G)=2\) (see Proposition \ref{proposition:tree_and_complete}, below).

Suppose for the sake of contradiction that there exists a tree-cut decomposition \(\mathcal{T}=(T,\mathcal{X})\) of \(G\) with \(w(\mathcal{T})=2\) and all nonempty bags \(X_b\in \mathcal{X}\) satisfying \(G[X_b]\) connected. By the argument from Proposition \ref{prop:empty_bags_trivalent} (which would not introduce disconnected bags), we may assume that \(\mathcal{T}\) has no empty bags of valence other than \(3\).

We first note that \(u_1\) and \(u_2\) must be contained in the same bag; otherwise any link on the unique path connecting their corresponding nodes would have adhesion at least \(3\). Since no bag can have more than two vertices, we know that one of our bags is \(X_{b_u}=\{u_1,u_2\}\).  An identical argument gives us bags \(X_{b_v}=\{v_1,v_2\}\) and \(X_{b_w}=\{w_1,w_2\}\).  Since there are no disconnected bags, each of \(x\) and \(y\) must be in their own bag, say \(X_{b_x}\) and \(X_{b_y}\), respectively.

Since \(\textrm{val}(x)=\textrm{val}(y)=3\), we know that neither \(b_x\) nor \(b_y\) is a leaf. Since \(\mathcal{T}\) has no leaves corresponding to empty bags, the only possible leaves are \(b_u\), \(b_v\), and \(b_w\).  Thus by Corollary \ref{corollary:number_of_leaves}, we may assume that there are either \(2\) leaves and no empty bags, or  \(3\) leaves and at most \(1\) empty bag.

In the case of \(2\) leaves, we have that \(T\) must be a path graph, without loss of generality its nodes ordered \((b_u,b_x,b_v,b_y,b_w)\) or \((b_u,b_x,b_y,b_v,b_w)\).  If \(T\) has \(3\) leaves, then it must be a (possibly trivial) subdivision of a star graph with \(3\) leaves, with at most one empty bag (which would necessarily correspond to the unique node of valence \(3\)). All these possibilities are illustrated in Figure \ref{fig:scw_not_2}. We use \(X_{b_\emptyset}\) to denote an empty bag, and illustrate several edges from \(E(G)\) for each decomposition.  Note that in each decomposition we reach a contradiction, with either a node or a link yielding width at least \(3\).  Having reached a contradiction, we may conclude that any tree-cut decomposition of \(G\) with width \(2\) has a disconnected bag.
\end{proof}
\begin{figure}[hbt]
    \centering
    \includegraphics[scale=0.7]{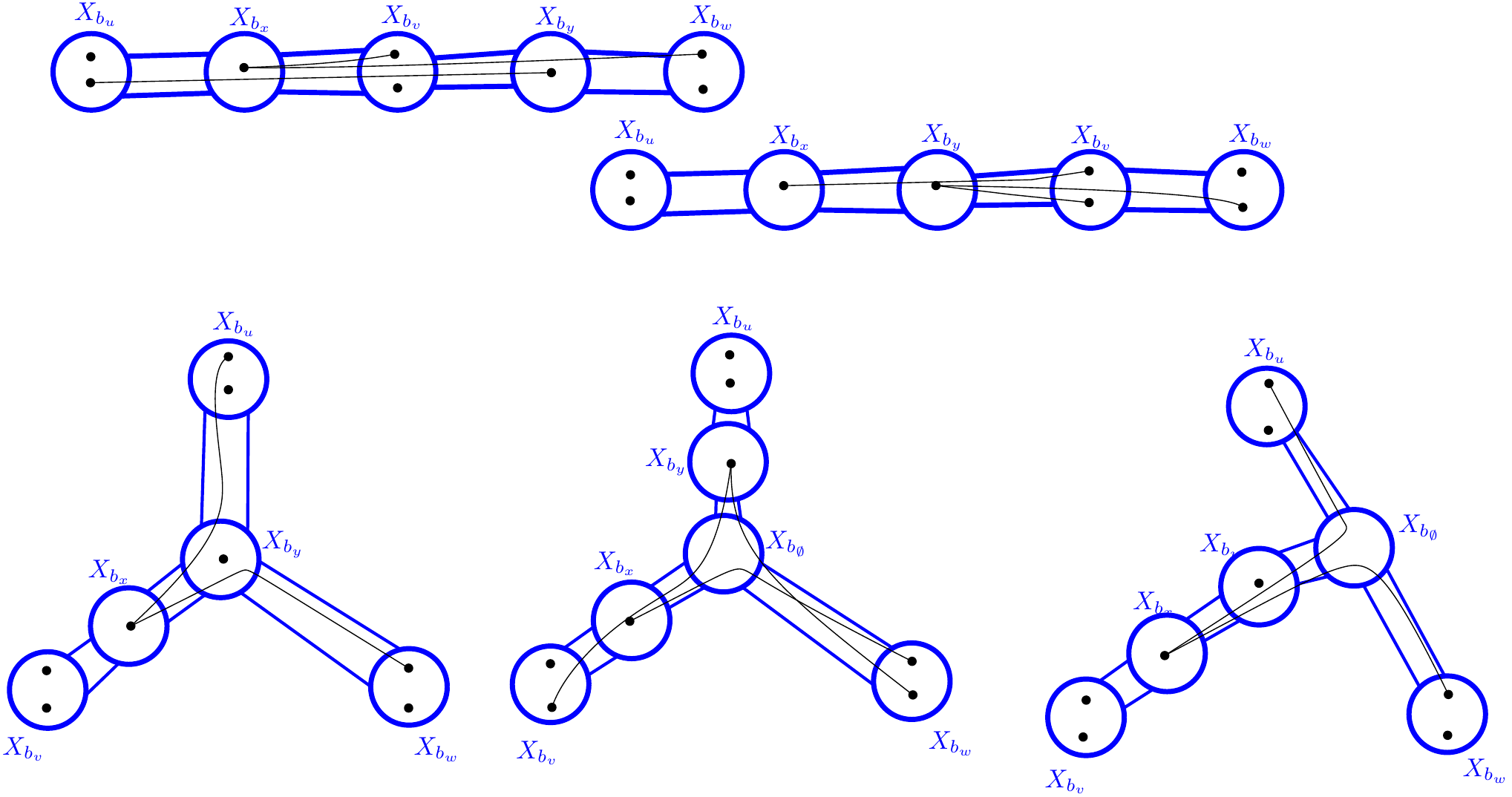}
    \caption{Possible decompositions of our graph \(G\), all with width at least \(3\).}
    \label{fig:scw_not_2}
\end{figure}

We now consider the behavior of screewidth under various graph theoretic operations.  We will later see in Lemma \ref{lemma:scw_not_minor_monotone} that screewidth can increase under taking minors of graphs.  However, this cannot happen under taking subgraphs, or under subdividing or refining edges.

\begin{proposition}\label{prop:subgraph}
Let $G$ be a graph and $H$ a subgraph of $G$. Then $\scw(H) \leq \scw(G)$.
\end{proposition}
\begin{proof}
Let $\mathcal{T} = (T, \mathcal{X})$ be an optimal tree-cut decomposition of $G$.  Consider the tree-cut decomposition \(\mathcal{T}'=(T,\mathcal{X}')\) on \(H\), where $\mathcal{X}' = \{X \cap V(H) \,|\, X \in \mathcal{X}\}$ with \(X'_b=X_b\cap V(H)\).  We claim that  $w(\mathcal{T}') \leq w(\mathcal{T})$. Since $\mathcal{T}'$ was obtained from $\mathcal{T}$ only by removing vertices and edges, each bag size, node adhesion, and edge adhesion must not have increased. Thus, we have a tree-cut decomposition on $H$ with width at most $\scw(G)$, so $\scw(H) \leq \scw(G)$. 
\end{proof}

\begin{lemma}\label{lemma:subdivision}
Let $G$ be a graph, and let $G'$ be a graph obtained from $G$ by subdividing an edge. Then $\scw(G') \leq \scw(G)$.
\end{lemma}
\begin{proof}

Let $\cT = (T, \mathcal{X})$ be a tree-cut decomposition of $G$ of width $\scw(G)$. Let $s = vw$ be an edge of $G$ to subdivide by adding a vertex $u$, and let $s_1 = vu$ and $s_2 = uw$ be the new edges after subdividing. 

We now deal with two cases.  First, suppose that $v$ and $w$ are in the same bag $X_b$. This means that \(w(\mathcal{T})\geq |X_b|+|\adh(b)|\geq 2\). Construct a tree-cut decomposition on \(G'\) be adding a leaf node \(c\) adjacent to \(b\) connected by a link \(l\), where \(X_c=\{u\}\). Then \(\adh(c)=\emptyset\), so \(|X_c|+|\adh{c}|=1< w(\mathcal{T})\).  Furthermore, \(\adh(l)\) consists of two edges, so \(|\adh(l)|\leq w(\mathcal{T})\). Since the rest of $\mathcal{T}'$ is identical to $\mathcal{T}$, it follows that the width of $\mathcal{T}'$ is equal to the width of $\mathcal{T}$.

For the second case, suppose that $v$ and $w$ are in different bags, say \(X_{b_v}\) and \(X_{b_w}\). Let \(b_v=b_{u_0}-b_{u_1}-\cdots-b_{u_k}=b_w\) be the unique path from \(b_v\) to \(b_w\) on \(T\), noting that we may have \(k=1\) if there are no intermediate vertices.  Let \(l\) be the link connecting \(b_v\) to \(b_{u_1}\), and create a tree-cut decomposition of \(G'\) by subdividing \(l\) with a node \(b_u\) where \(X_{b_u}=\{u\}\), say with \(l'\) connecting \(b_v\) to \(b_u\) and \(l''\) connecting \(b_u\) to \(b_{u_1}\).  Note that \(\adh(b_u)=\adh(l)-\{s\}\), so  \(|X_{b_u}|+|\adh(b_u)|=1+|\adh(l)|-1=|\adh(l)|\).  Moreover, \(\adh(l')=(\adh(l)-\{s\})\cup \{s_1\}\), and  \(\adh(l'')=(\adh(l)-\{s\})\cup \{s_2\}\).  No other bag sizes change, and the only possible change to a node or link adhesion is replacing \(s\) with \(s_i\) for some \(i\).  Thus \(w(\mathcal{T}')=w(\mathcal{T})\).

We conclude that \(\scw(G')\leq \scw(G)\). 
\end{proof}

\begin{lemma}\label{lemma:smoothing}
Let $G$ be a graph, and let $G'$ be a graph obtained from $G$ by smoothing a 2-valent vertex. Then $\scw(G') \leq \scw(G)$.
\end{lemma}
\begin{proof}
Let $\cT = (T, \mathcal{X})$ be a tree-cut decomposition of $G$ of width $\scw(G)$, and let $v-u-w$ be a path on \(3\) vertices of $G$ to refine, with $u$ being a $2$-valent vertex. Let \(s_1\) and \(s_2\) be the edges connecting \(v\) to \(u\) and \(u\) to \(w\) in \(G\), and let \(s\) be the new edge connecting  \(v\) to \(w\) in \(G'\). Let \(\cT'\) be the tree-cut decomposition of \(G'\) obtained by deleting \(u\) from its bag.

First we assume that at least two of \(v,u,w\) are in the same bag.
If $u$ is in the same bag as at least one of $v$ and $w$, say \(v\), then none of the adhesion sets change, except possibly swapping \(s_2\) for \(s\); and the size of the bag containing $u$ decreases by one. Alternatively, if \(v\) and \(w\) are in the same bag and \(u\) is in another, then smoothing at \(u\) will decrease a bag size and remove \(s_1\) and \(s_2\) from any adhesions, without introducing \(s\) to any.  Thus in this case we have \(w(\mathcal{T}')\leq w(\mathcal{T})\).
For the second case, assume that $v,u,w$ are all in different bags, say \(X_{b_v}\), \(X_{b_u}\), and \(X_{b_w}\). Without loss of generality, we may assume \(b_w\) does not lie on the path in $T$ from \(b_v\) to \(b_u\).  When we delete \(s_1\) and \(s_2\) and replace them with \(s\), any adhesion set that loses \(s_1\) or \(s_2\) will either stay the same size or decrease, since at most it would gain \(s\).  The only other change to an adhesion is that \(\adh(b_u)\) might gain the edge \(s\), but \(X_{b_u}\) loses a vertex, meaning \(|X_{b_u}|+|\adh(b_u)|\) cannot increase when passing from \(\mathcal{T}\) to \(\mathcal{T}'\). Hence \(w(\mathcal{T}')\leq w(\mathcal{T})\).

We conclude that \(\scw(\mathcal{T}')\leq \scw(\mathcal{T})\).
\end{proof}

\begin{corollary}
Screewidth is invariant under subdivision and smoothing.
\end{corollary}

\begin{proof}
If \(G\) can be obtained from \(G'\) by subdivision, then \(G'\) can be obtained from \(G\) by smoothing.  Lemma \ref{lemma:subdivision} tells us \(\scw(G')\leq \scw(G)\), and Lemma \ref{lemma:smoothing} tells us \(\scw(G)\leq \scw(G')\).  Thus \(\scw(G)=\scw(G')\).
\end{proof}

For our final result, we consider the behavior of screewidth under the deletion of a \emph{bridge}; that is, an edge that when deleted disconnects the graph.

\begin{proposition}\label{proposition:scw_bridge}
Let \(G\) be a graph with a bridge \(e\), such that the two components of \(G-e\) are \(G_1\) and \(G_2\).  Then \(\scw(G)=\max\{\scw(G_1),\scw(G_2)\}\).
\end{proposition}

\begin{proof}
Since \(G_1\) and \(G_2\) are subgraphs of \(G\), we have \(\max\{\scw(G_1),\scw(G_2)\}\leq \scw(G)\) by Proposition \ref{prop:subgraph}.  For the other direction, let \(\mathcal{T}_1=(T_1,\mathcal{X}_1)\) and  \(\mathcal{T}_2=(T_2,\mathcal{X}_2)\) be optimal decompositions of \(G_1\) and \(G_2\), respectively. Let \(e=uv\) with \(u\in G_1\) and \(v\in G_2\), and let \(u\in B_{b}\in\mathcal{X}_1\) and \(v\in B_c\in\mathcal{X}_2\) where \(b\in V(T_1)\) and \(c\in V(T_2)\)  Consider \(\mathcal{T}=(T,\mathcal{X}_1\cup\mathcal{X}_2)\), where \(T\) consists of \(T_1\) and \(T_2\) with a new link \(l\) connecting \(b\) and \(c\).  The adhesion of \(l\) is \(1\), and the link adhesions, node adhesions, and bag sizes are otherwise unchanged.  Thus we have \(\lw(\mathcal{T})=\max\{\lw(\mathcal{T}_1),\lw(\mathcal{T}_2)\}\) and \(\bw(\mathcal{T})=\max\{\bw(\mathcal{T}_1),\bw(\mathcal{T}_2)\}\), implying that \(w(\mathcal{T})=\max\{w(\mathcal{T}_1),w(\mathcal{T}_2)\}=\max\{\scw(G_1),\scw(G_2)\}\).  This completes the proof.
\end{proof}

\section{Scramble number and screewidth}\label{section:sn_and_scw}

We now prove our main theorem connecting scramble number and screewidth: that \(\sn(G)\leq \scw(G)\) for any graph \(G\).

\begin{proof}[Proof of Theorem \ref{theorem:main}]
Let $G$ be a graph with a scramble $\cS$ such that $||\cS||=n$. Suppose for the sake of contradiction that there exists a tree-cut decomposition $\cT=(T,\mathcal{X})$ such that $w(\cT)< n$. Note that \(w(\cT)<||\mathcal{S}||\leq h(\mathcal{S})\leq |V(G)|\), so \(\cT\) is not the trivial tree-cut decomposition, and thus has at least one link. Let $l$ be a link in $T$. If we delete $l$, we split $T$ into two components. Let \(A\) be the union of the bags corresponding to the nodes of one component; and \(B\) the union of the bags corresponding to the nodes in the other component. Note that $A$ and $B$ partition $V(G)$

Suppose there exist eggs $E_1,E_2\in \cS$ such that $E_1\subseteq A$ and $E_2\subseteq B$. Then $\adh(l)$ is an egg cut of $\cS$ and therefore has at least $n$ elements. Thus $w(\cT)\geq \textrm{aw}(\cT)\geq \textrm{adh}(l)\geq n$ contradicting $w(\cT)<n$.

Now suppose that neither \(A\) nor \(B\) contains an egg from $\cS$. That implies that for every egg \(E_i\) in $\cS$, the graph \(G[E_i]\) has an edge in $\adh(l)$. Note that if we take one vertex from each edge in $\adh(l)$, we have a hitting set on $\cS$, which must be at least as big as $n$. Thus $w(\cT)\geq \lw(\cT) \geq \adh(l) \geq n$, again contradicting w$(\cT) < n$.

Thus for any link $l \in E(T)$, if we delete $l$ in $T$, $V(G)$ is partitioned into two subsets, one containing an egg and the other not. Thus we can orient every edge in $T$ to point to the component of $T$ whose bags contain an egg. Because trees have fewer edges than vertices, there exists a node $b\in V(T)$ that has outdegree zero under this orientation. Suppose there exists an egg $E_i\in \cS$ that has no vertices in $X_b$ and no edges in $\adh(b)$. Then we can delete one of the links adjacent to $b$ to define a partition on $V(G)$ where the part that does not contain the vertices in $X_b$ contains $E_i$.  Therefore, the link deleted is directed towards $E_i$, and $b$ has positive outdegree, a contradiction. Thus every egg in $\cS$ has either a vertex in $X_b$ or an edge in $\adh(b)$. This means we can choose a vertex from each edge in $\adh(b)$, which together with the set $X_b$ to yields hitting set for $\mathcal{S}$. We now see that $w(\cT)\geq \bw(\cT)\geq |X_v| + |\adh(v)|\geq n$ contradicting $w(\cT)<n$.

We conclude that there cannot exist a scramble $\cS$ and a tree-cut decomposition $\cT$ such that $||\cS||>w(\cT)$. This proves that $\sn(G)\leq \scw(G)$.
\end{proof}

\begin{example}
We now present an example of the usefulness of this result, first suggested to the authors by Ben Baily. Consider consider the graph illustrated in Figure \ref{figure:easier_sn_2_argument}, along with a tree-cut decomposition of width \(2\).  This graph appeared in \cite[Example 2.8]{echavarria2021scramble} as an example of a graph of scramble number \(2\) with an immersion minor of scramble number \(3\).  The argument presented there that the graph has scramble number equal to \(2\) involved a cumbersome argument that checked many cases.  Now armed with Theorem \ref{theorem:main}, and noting that a scramble consisting of (any) two disjoint eggs on the graph has order \(2\), we can quickly argue that \(2\leq \sn(G)\leq \scw(G)\leq 2\), and so \(\sn(G)=2\).
\end{example}

\begin{figure}[hbt]
    \centering
    \includegraphics{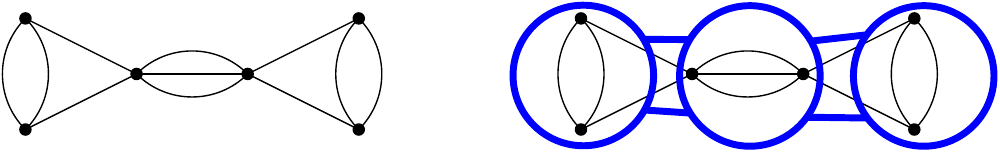}
    \caption{A graph of screewidth and scramble number \(2\).}
    \label{figure:easier_sn_2_argument}
\end{figure}

We can use the relationship between scramble number and screewidth to deepen our understanding of screewidth over the next several results.  

\begin{proposition}\label{proposition:tree_and_complete}
For any graph $G$, we have $\scw(G) = 1$ if and only if $G$ is a tree. For a simple graph \(G\) on \(n\) vertices, we have \(\scw(G)=n-1\) if and only if \(G=K_n\).
\end{proposition}

\begin{proof}
If $G$ is a tree, then define a tree-cut decomposition $\mathcal{T} = (T, \mathcal{X})$ by $T = G$ and $X_v = \{v\}$ for each $v \in V(G)$. Then $\mathcal{T}$ has width 1, since every edge adhesion is of size 1, every node adhesion is empty, and every bag has size \(1\). Any tree-cut decomposition has positive width, so \(\scw(G)=1\). Conversely, if \(G\) is a graph with \(\scw(G)=1\), then we have \(1\leq \sn(G)\leq \scw(G)=1\), implying \(\sn(G)=1\).  By \cite[Corollary 4.2]{scramble}, \(G\) is a tree.

Now let \(G\) be a simple graph. If \(G\) is a complete graph, then \(\scw(G)\geq \sn(G)\geq n-1\), due to the scramble whose eggs are precisely the singleton sets \(\{v\}\) for \(v\in V(G)\). Since \(G\) is simple, by Proposition \ref{prop:scw_n-1} we have \(\scw(G)\leq n-1\), so \(\scw(G)=n-1\).  Conversely,  if \(\scw(G)=n-1\), then since \(\scw(G)\leq n-\alpha(G)\) by Lemma \ref{lemma: scw_n-alpha} we have \(\alpha(G)=1\), so \(G\) is a complete graph.
\end{proof}

\begin{lemma}\label{lemma:scw_not_minor_monotone}
Screewidth is not minor monotone.  That is, it is possible for a graph \(G\) to have a minor \(H\) with \(\scw(H)>\scw(G)\).
\end{lemma}

\begin{figure}[hbt]
    \centering
    \includegraphics[scale=.4]{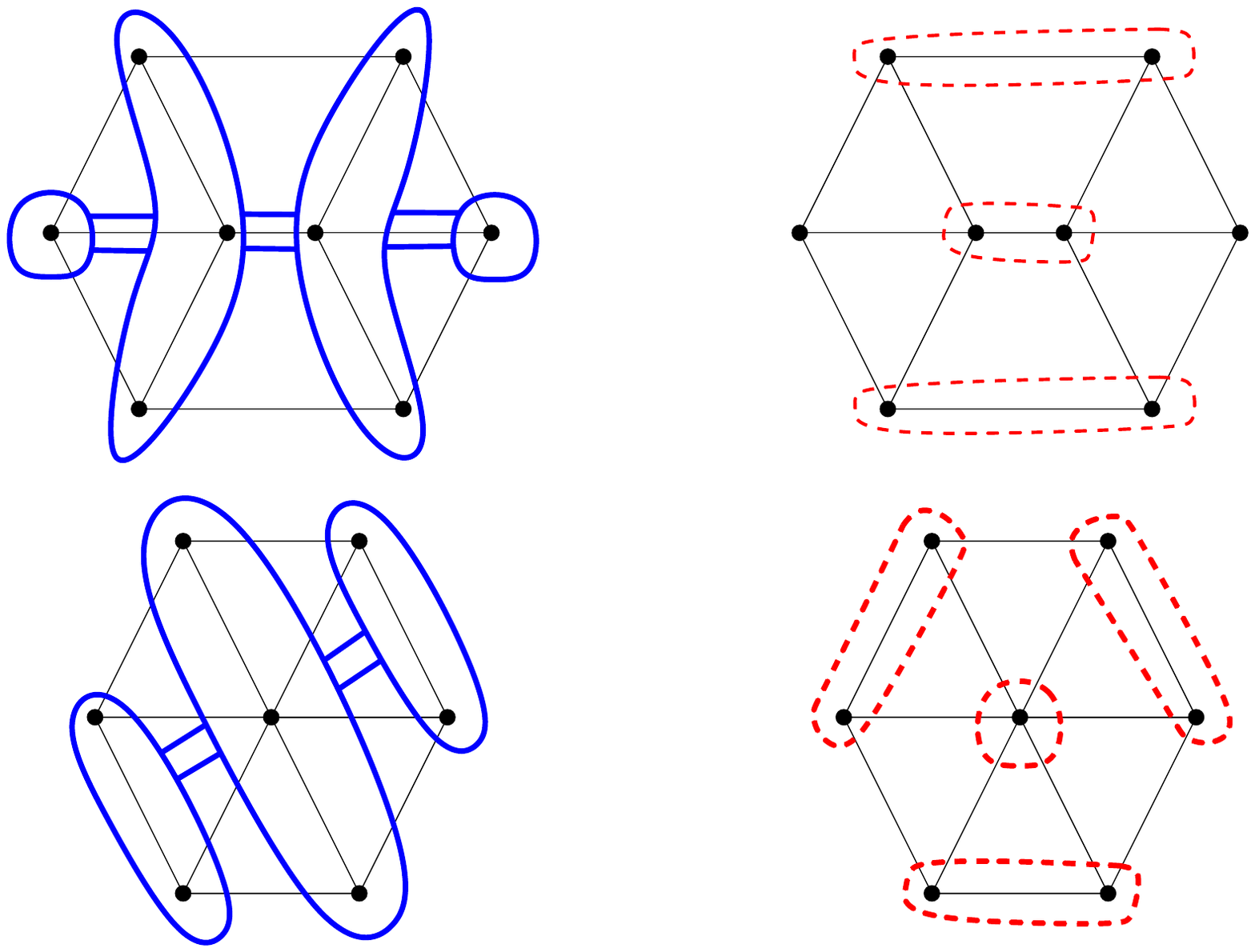}
    \caption{Graph demonstrating screewidth is not minor monotone.}
    \label{fig:minor_monotone}
\end{figure}

\begin{proof}
Consider the graph \(G\) pictured on the top of Figure \ref{fig:minor_monotone}, on the left with a tree-cut decomposition of width \(3\) and on the right with a scramble of order \(3\); and the minor \(H\) of \(G\) pictured below, on the left with a tree-cut decomposition of width \(4\) and on the right with a scramble of order \(4\).  It follows that \(\scw(G)=3\) and \(\scw(H)=4\), so screewidth is not minor monotone.
\end{proof}
We remark that these same graphs were used in \cite[Example 4.4]{scramble} to prove that scramble number is not minor monotone.
We now present a family of graphs for which we have \(\sn(G)=\scw(G)\).

\begin{lemma}
If $G$ is a simple graph on $n$ vertices and minimum valence $\delta(G) \geq \lfloor \frac{n}{2} \rfloor + 1$, then $\scw(G) = \sn(G) = n - \alpha(G)$. 
\end{lemma}

\begin{proof}
It is shown in \cite{echavarria2021scramble} that $\sn(G) = n - \alpha(G)$ whenever $G$ is simple with $\delta(G) \geq \lfloor \frac{n}{2} \rfloor + 1$. Together with Theorem \ref{theorem:main} and Lemma \ref{lemma: scw_n-alpha}, we have $n - \alpha(G) = \sn(G) \leq \scw(G) \leq n - \alpha(G)$, giving the desired equality.
\end{proof}

This allows us to prove a result in computational complexity.  Consider the following problem.

\begin{itemize}
  \item[] \textsc{Screewidth}
  \item[] \textbf{Input:} A graph \(G=(V,E)\) and an integer \(k\).
  \item[] \textbf{Question:} Is \(\scw(G)\leq k\)?
\end{itemize}

\begin{corollary} The problem \textsc{Screewidth} is NP-complete.
\end{corollary}

\begin{proof} First we show that \textsc{Screewidth} is in NP.  As a ``yes'' certificate, take a tree-cut decomposition \(\mathcal{T}\) of \(G\) with width at most \(k\).  Computing link adhesions, node adhesions, and bag sizes can be accomplished in polynomial time, so this is indeed a polynomial time certificate.

To prove \textsc{Screewidth}  is NP-hard, we follow the construction from \cite[\S 3]{echavarria2021scramble}.  Given a simple connected graph \(G\), build \(\hat{G}\) by adding \(|V(G)|\) vertices and connecting the new vertices to all \(2|V(G)|-1\) other vertices in the new graph.  Note that \(\delta(\hat{G})\geq \left\lfloor\frac{|V(\hat{G})|}{2}\right\rfloor+1\), so \(\scw(\hat{G})=2n-\alpha(\hat{G})\).  By the proof of  \cite[Lemma 3.5]{echavarria2021scramble}, we know \(\alpha(\hat{G})=\alpha(G)\), so we have \(\scw(\hat{G})=2n-\alpha({G})\).  Thus an efficient algorithm for upper bounding \(\scw(\hat{G})\) would provide an efficient algorithm for lower bounding \(\alpha(G)\) with only polynomial increase in the size of the input graph.  Since it is NP-hard to lower bound \(\alpha(G)\), we conclude that it is NP-hard to compute screewidth.
\end{proof}

Since screewidth is invariant under subdivisions, the problem \textsc{Screewidth} is NP-hard even when restricting to bipartite graphs, or to graphs with arbitrarily large (fixed) girth.  We remark that since \(\tw(G)\leq \sn(G)\leq \scw(G)\), any problem that can be solved in polynomial time for graphs of bounded treewidth can also be solved in polynomial time for graphs of bounded screewidth.  It would be interesting to investigate whether there are any problems that can be solved for graphs of bounded screewidth that are thought to be hard for graphs of bounded treewidth.

There do exist graphs $G$ such that $\sn(G) < \scw(G)$. As a first simple example, consider the graph $K_3 \circ B_{2,3}$, a triangle with each of its vertices connected to an additional vertex by three parallel edges (often called a \emph{banana}) as illustrated in Figure \ref{fig:scw_neq_sn}.

\begin{figure}[hbt]
    \centering
    \includegraphics[scale=.5]{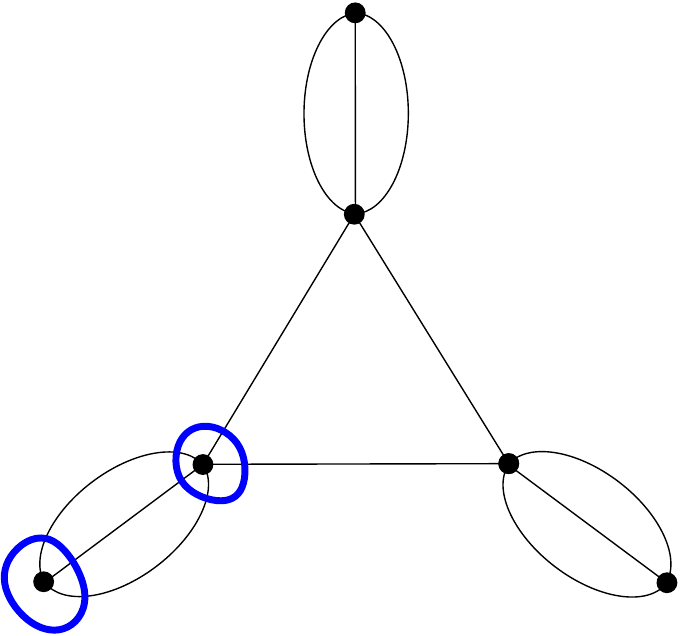}\quad\quad
    \includegraphics[scale=.5]{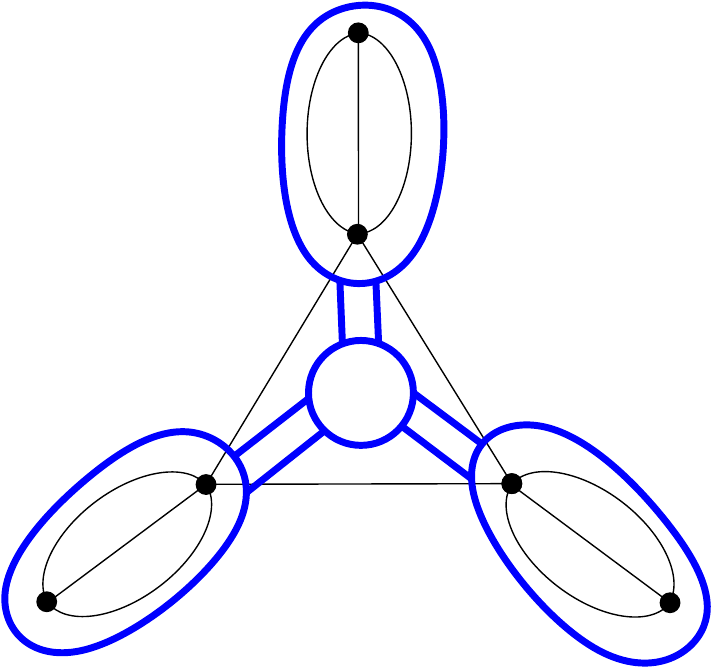}
    \caption{A graph with $\sn(G) < \scw(G)$.}
    \label{fig:scw_neq_sn}
\end{figure}

\begin{proposition}
We have \(\sn(K_3 \circ B_{2,3})=2\) and \(\scw(K_3 \circ B_{2,3})=3\).
\end{proposition}

\begin{proof}

The decomposition pictured on the right in Figure \ref{fig:scw_neq_sn} shows \[\scw(K_3 \circ B_{2,3})\leq 3.\] Suppose for the sake of contradiction that \(K_3 \circ B_{2,3}\) has a tree-cut decomposition \(\mathcal{T}=(T,\mathcal{X})\) with width  \(2\). Without loss of generality, \(\mathcal{T}\) has no empty bags with valence smaller than \(3\). Three of the bags must be the pairs of vertices in each banana:  putting the vertices in separate bags would give a link of adhesion at least \(3\), and having any more vertices would increase the bag-width to at least \(3\).  If there are no empty bags, then \(T\) is a path on three bags, giving a bag-width of \(3\).  Thus there must be an empty bag, whose valence is at least \(3\), implying that there must be at least three leaves in \(T\). Since only a nonempty bag can be a leaf, there are exactly three leaves, and \(T\) must be the star graph on four vertices.  It follows that the decomposition must in fact be the one pictured in Figure \ref{fig:scw_neq_sn}, which as noted has width \(3\), a contradiction.  Thus \(\scw(K_3 \circ B_{2,3})= 3\).

It remains to show that \(\sn(K_3 \circ B_{2,3})=2\).
The scramble illustrated on the left in Figure \ref{fig:scw_neq_sn} has \(h(\mathcal{S})=2\) and \(e(\mathcal{S})=3\), for a scramble of order \(2\).  Thus \(\sn(K_3 \circ B_{2,3})\geq 2\).  Suppose for the sake of contradiction that \(K_3 \circ B_{2,3}\) has a scramble \(\mathcal{S}\) with \(||\mathcal{S}||\geq 3\).  Since we have \(e(\mathcal{S})\geq 3\), and we can disconnect any banana from the graph by deleting two edges, it is impossible to have both an egg in a banana and an egg in the complement of that banana.  If no egg is contained in any banana, then every egg contains two of the vertices from the central triangle, meaning that there is a hitting set of size \(2\), a contradiction.  Thus some egg must be contained in a banana, and every other egg must intersect that same banana.  This means that the \(2\) vertices of the banana form a hitting set, another contradiction.  We conclude that \(\sn(K_3 \circ B_{2,3})=2\), completing the proof.
\end{proof}

The next sequence of results shows that screewidth can grow quadratically in scramble number.  We start with several lemmas concerning trees.  A \emph{leaf-to-leaf geodesic} in a tree is a path connecting two distinct leaves.  A \emph{forest} is a graph all of whose connected components are trees.

\begin{lemma}
Let \(T\) be a tree with \(n\geq 4\) leaves and no \(2\)-valent nodes. There exists a node \(v\) such that \(T-v\) consists of forests, each including at most \(\lfloor n/2\rfloor \) of the original leaves.  We call any such node a \emph{leaf-centroid} of \(T\).
\end{lemma}

\begin{proof}
  Let \(v\in V(T)\); if \(v\) is a leaf-centroid of \(T\), we are done.  If not, there is a component \(C\) of \(T-v\) with more than \(\lfloor n/2\rfloor\) of the tree's leaves in it.  There is a unique vertex of \(C\) incident to \(v\), call this \(v'\).  Replace \(v\) with \(v'\).  If \(v'\) is a leaf-centroid, we are done; if not, repeat.  We claim this process will eventually lead us to a leaf-centroid of \(T\).

First, moving from \(v\) to \(v'\) cannot connect any vertices in \(T-C\) to vertices in \(C\), and in \(T-C\) there are strictly fewer than \(n/2\) leaves of \(T\); this means that we will not double-back from \(v'\) to \(v\), and will instead go to a new vertex.  Second, since \(T\) has no \(2\)-valent nodes, moving from \(v\) to \(v'\) disconnects some leaves from \(C\), meaning that the largest number of leaves in a component must strictly decrease. This process cannot continue forever, as eventually we will get down to at most \(\lfloor n/2\rfloor\) leaves in any component. 
\end{proof}

\begin{corollary}\label{corollary:min_geodesics}
Let \(T\) be a tree with \(n\geq 4\) leaves and all other vertices trivalent.  There exists a node \(v\in T\) such that at least \[\left\lceil\frac{1}{4}n(n+2)\right\rceil-1\] leaf-to-leaf geodesics pass through \(v\).  Moreover, for every \(n\), there exists a tree with \(n\) leaves, namely the cubic caterpillar graph, such that this is the largest number of leaf-to-leaf geodesics through any node \(v\in T\).
\end{corollary}

\begin{proof}
Let \(v\) be a leaf-centroid of \(T\).  Since \(T\) has all non-leaves trivalent, \(T-v\) consists of three components, none of which have more than \(\lfloor n/2\rfloor\) leaves.  Thus there exist \(3\) disjoint sets of leaves, say with \(L_1\), \(L_2\), and \(L_3\) leaves, such that all geodesics connecting leaves in different sets pass through \(v\).  It follows that there are \(L_1L_2+L_2L_3+L_3L_1\) leaf-to-leaf geodesics passing through \(v\).

Assume for the moment that \(n\) is even.  Then we wish to minimize \(L_1L_2+L_2L_3+L_3L_1\) subject to the constraints  \(1\leq L_i\leq n/2\) for all \(i\), and \(L_1+L_2+L_3=n\). Replace for the moment \(L_1\), \(L_2\), and \(L_3\) with the continuous variables \(x\), \(y\), and \(z\).  We then wish to minimize \(xy+yz+zx\) subject to the constraint \(x+y+z=n\) and \(1\leq x,y,z\leq n/2\).  Since \(2(xy+yz+zx)=n^2-x^2-y^2-z^2\), this is the same as maximizing the distance from the origin to our region in \(xyz\)-space. This will occur at one of the geometric vertices of our region, which is in general a hexagon on the plane \(x+y+z=n\).  The vertices are the six permutations of \((n/2,n/2-1,1)\). Thus the function \(xy+yz+zx\) is minimized at an integral point, with minimum \(\frac{1}{4}n(n+2)-1\).  Therefore the leaf centroid has at least \(\frac{1}{4}n(n+2)-1\) leaf-to-leaf geodesics passing through it.

The analysis for \(n\) odd is similar, except the vertices are permutations of \(((n-1)/2,(n-1)/2,1)\).  This yields a minimum of $\frac{1}{4}n(n+2)-\frac{3}{4}$, which is equal to to \(\left\lceil\frac{1}{4}n(n+2)-1\right\rceil\).  Thus the leaf centroid has at least \(\left\lceil\frac{1}{4}n(n+2)\right\rceil-1\) leaf-to-leaf geodesics passing through it.

Now consider the cubic caterpillar tree with \(n\) leaves.  We refer to the non-leaf nodes as \(v_2,\ldots,v_{n-1}\), arranged in a path.   Note that deleting \(v_k\) yields three collections of vertices with sizes \(1\), \(k-1\), and \(n-k\).  It follows that the number of leaf-to-leaf geodesics passing through \(v_k\) is equal to
\[(k-1)(n-k)+(n-k)+(k-1)=(k-1)(n-k)-(n-1).\]
Treating \(k\) as a continuous variable, this expression is concave down in \(k\), and has first derivative in \(k\) vanishing when \(k=(n+1)/2\).  Thus when \(n\) is odd, the integer \(k\) maximizing this expression is \(k=(n+1)/2\), yielding an expression of
\[(k-1)(n-k)-(n-1)=\frac{1}{4}n(n+2)-\frac{3}{4}.\]
When \(n\) is even, the integer \(k\) maximizing this expression is either \(k=n/2\) or \(k=n/2+1\).  Both of these yield
\[(k-1)(n-k)-(n-1)=\frac{1}{4}n(n+2)-1.\]
In both the even and the odd cases, the formulas match our claim.  Thus every node in a cubic caterpillar graph with \(n\) leaves has at most \(\left\lceil\frac{1}{4}n(n+2)\right\rceil-1\) leaf-to-leaf geodesics passing through it.
\end{proof}

Our next theorem shows that screewidth can grow quadratically in scramble number.
\begin{theorem}\label{theorem:scw_quadratic_sn}
For \(n\geq 4\), there exists a graph \(G\) with \(\sn(G)=n\) and \(\scw(G)=\left\lceil\frac{1}{4}n(n+2)\right\rceil-1\).
\end{theorem}

\begin{figure}[hbt]
    \centering
    \includegraphics{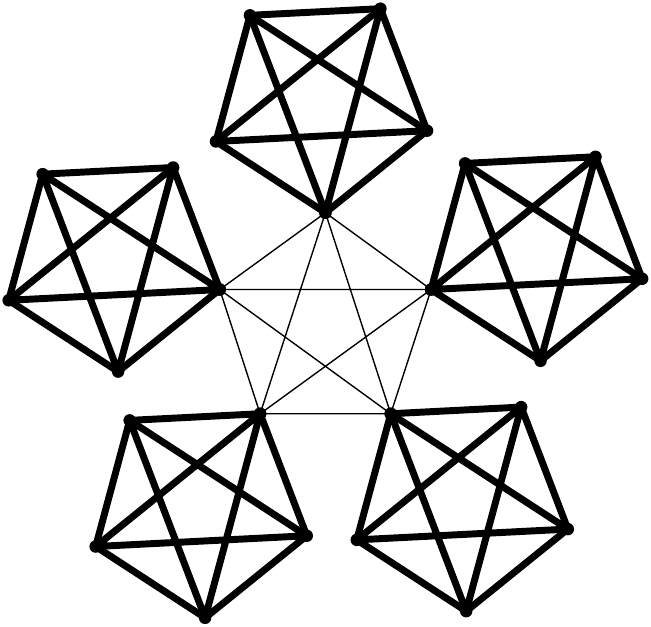}
    \caption{The graph \(G\) from Theorem \ref{theorem:scw_quadratic_sn} where \(n=5\).  The bold edges represent \(\binom{5}{2}+1=11\)} parallel edges each.
    \label{fig:scw_quadratic_sn}
\end{figure}

\begin{proof}
Let $K^m_n$ denote the complete multigraph on $n$ vertices, with $m$ multiedges between every pair of vertices, where \(m=\binom{n}{2}+1\).
Let $G$ be the rooted product of $K_n$ with $K^m_n$; this graph is illustrated in Figure \ref{fig:scw_quadratic_sn} for \(n=5\). We refer to each copy of \(K^m_n\) as a bulb. First, we claim that $\sn(G) = n$. Since \(G\) has \(K_n^m\) as a subgraph, certainly the scramble number is at least this large: the scramble on \(K_n^m\) with each vertex in its own egg has order \(n\).  Suppose for the sake of contradiction that \(\mathcal{S}\) is a scramble of order at least \(n+1\) on \(G\).  First we note that since \(e(\mathcal{S})\geq n+1\), there cannot exist a pair of eggs \(E_1\) and \(E_2\) where \(E_1\) is contained in a bulb and \(E_2\) is contained in its complement.  If no egg is contained in any bulb then every egg has at least two vertices in the central \(K_n\), giving a hitting set of size \(n-1\), a contradiction.  Thus at least one egg is contained in some bulb, and every other egg intersects that bulb.  This means that there is a hitting set of size \(n\), a contradiction.  We conclude that \(\sn(G)=n\).

Let \(\mathcal{T}=(T,\mathcal{X})\) be an optimal tree-cut decomposition of \(G\), with all empty bags of valence three.  Considering the star graph decomposition \(\mathcal{T}^*\) with each copy of \(K_n^m\) in a bag corresponding to a leaf node and a node corresponding to an empty bag in the middle, we have that \( w(\mathcal{T}^*)= |E(K_n)|=\binom{n}{2}\).  It follows that in \(\mathcal{T}\), all vertices in a given bulb must be in the same bag; otherwise some link would have adhesion at least as large as \(m>\binom{n}{2}\), giving \(w(\mathcal{T})> w(\mathcal{T}^*)\) and contradicting optimality.

We now argue that in our optimal tree-cut decomposition, we may assume that every non-empty bag must consist of exactly one bulb, and correspond to a leaf node.  Suppose not, so that there is some non-empty bag \(X_b\) such that either \(X_b\) contains multiple bulbs, or \(b\) is not a leaf node, or both.  Modify the tree-cut decomposition from \(\mathcal{T}=(T,\mathcal{X})\) to \(\mathcal{T}'=(T',\mathcal{X}')\) by adding a link \(l\) to \(b\) connecting it to a new leaf \(c\), and moving one of the bulbs from \(X_b\) to \(X_c\); we refer to \(b\) in this new tree as \(b'\). We claim that \(w(\mathcal{T}')\leq w(\mathcal{T})\).

First note that the bag sizes, node adhesions, and link adhesions in these tree-cut decompositions are largely the same, except for \(b\) in \(\mathcal{T}\), and \(b'\), \(l\), and \(c\) in \(\mathcal{T'}\).  Next, note that \(|X_c|+|\adh(c)|=n+0=n\leq|X_b|\leq |X_b|+|\adh(b)|\), that \(\adh(l)=n-1<|X_b|\leq |{X_b}|+|\adh({b})|\). Finally, we have \(|X_{b'}|=|X_b|-n\) and \(|\adh(b')|=|\adh(b')|+(n-1)\), so \(|X_{b'}|+|\adh(b')|<|X_b|+|\adh({b})|\).  Thus we have \(w(\mathcal{T}')\leq w(\mathcal{T})\).

Henceforth we assume that \(\mathcal{T}=(T,\mathcal{X})\) has each bulb as its own bag \(X_b\) corresponding to a leaf \(b\), each of which has \(|X_b|+|\adh(b)|=n\). We may also assume that every empty bag corresponds to a trivalent node. Consider a non-leaf \(b\) of \(\mathcal{T}\), which must correspond to an empty bag by assumption.  Since the bulbs of \(G\) are connected in a complete graph, we have \(\adh(b)\) is equal to the number of leaf-to-leaf geodesics in \(T\) that pass through the bag \(b\).  By Corollary \ref{corollary:min_geodesics}, some node has at least \[\left\lceil\frac{1}{4}n(n+2)\right\rceil-1\]
edges as part of its adhesion.  This gives us that
\[\scw(G)\geq \left\lceil\frac{1}{4}n(n+2)\right\rceil-1.\]
To show equality, consider the tree-cut decomposition of \(G\) that is the caterpillar graph with \(n\) leaf nodes whose bags are the bulbs; this is illustrated for \(n=5\) in Figure \ref{figure:caterpillar_tree_cut}.   For every leaf node \(b\), we have \(|X_b|+|\adh(b)|=n\).  For every non-leaf node \(b\), we have \(|X_b|+|\adh(b)|\leq \left\lceil\frac{1}{4}n(n+2)\right\rceil-1\), with equality for at least one node.  Finally, the adhesion of any link is equal to the number of leaf-to-leaf geodesics that pass through that link; this is at most the adhesion of the node or nodes incident to it that correspond to an empty bag.  Since \(G\) has a tree-cut decomposition of width equal to the upper bound, we have 
\[\scw(G)= \left\lceil\frac{1}{4}n(n+2)\right\rceil-1,\]
as claimed.
\end{proof}

\begin{figure}
    \centering
    \includegraphics[scale=0.7]{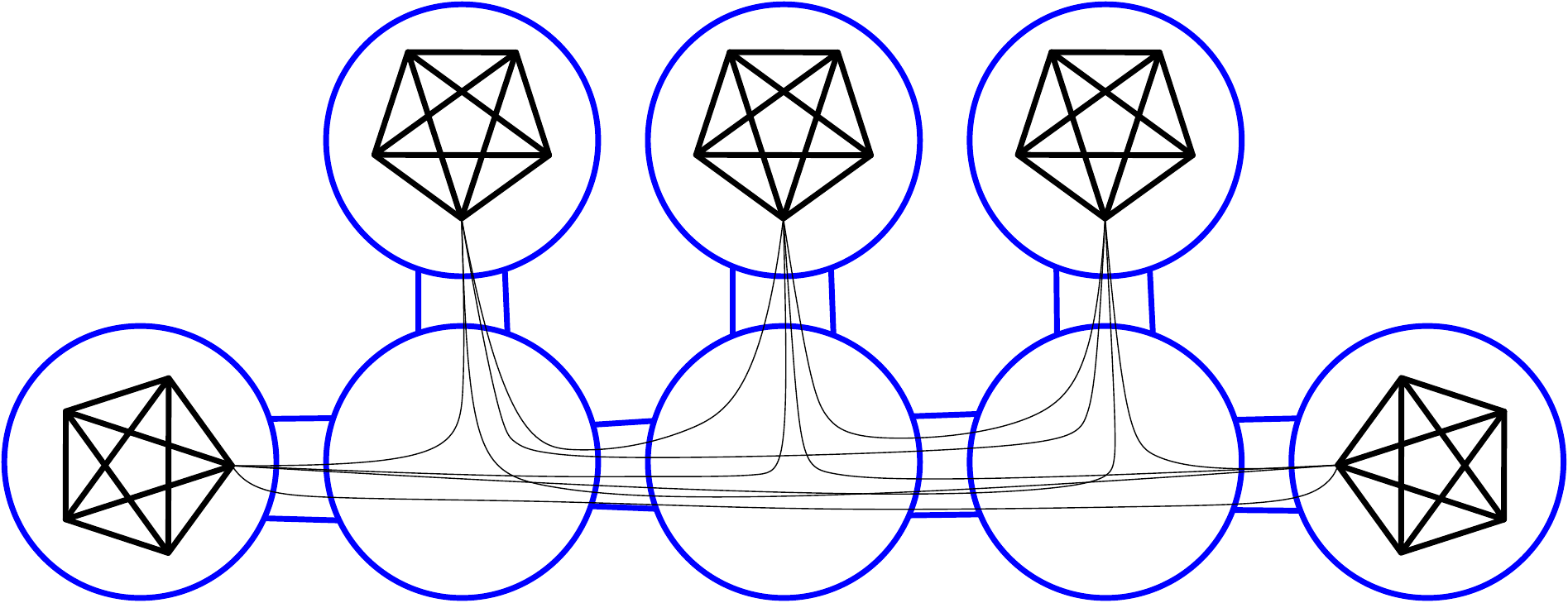}
    \caption{A tree-cut decomposition of \(G\) for \(n=5\), with width \(8\) due to to the adhesion of the central node.}
    \label{figure:caterpillar_tree_cut}
\end{figure}

It follows that the gap between scramble number and screewidth can be arbitrarily large.  In fact, the gap between scramble number and screewidth can be any nonnegative integer, not just \(\left\lceil\frac{1}{4}n(n+2)\right\rceil-1-n\) for some \(n\).  In particular, for a gap of \(\ell\), choose \(n\) so that \(\left\lceil\frac{1}{4}n(n+2)\right\rceil-1-n\geq \ell\).  Then, letting \(G\) be as in the proof of Theorem \ref{theorem:scw_quadratic_sn}, construct \(H\) by attaching a copy of \(K_{t+1}\) via a bridge to any vertex of \(G\), where \(t=\left\lceil\frac{1}{4}n(n+2)\right\rceil-1-\ell\).  By \cite[Lemma 2.4]{echavarria2021scramble}, the scramble number of this graph is equal to the maximum of \(\sn(G)=n\) and \(\sn(K_{t+1})=t\), which by construction is \(t\); similarly, by Proposition \ref{proposition:scw_bridge}, the screewidth is the maximum of \(\scw(G)\) and \(\scw(K_{t+1})=t\), which is \(\left\lceil\frac{1}{4}n(n+2)\right\rceil-1\).  The gap between \(\sn(H)\) and \(\scw(H)\) is thus \(\left\lceil\frac{1}{4}n(n+2)\right\rceil-1-t=\ell\).

It is unknown, however, whether screewidth can be arbitrarily large for some fixed scramble number.

\begin{question}
Is there a function \(f:\mathbb{N}\rightarrow\mathbb{N}\) such that \(\scw(G)\leq f(\sn(G))\) for all graphs \(G\)?
\end{question}

\section{Screewidth and gonality}\label{section:gonality}

Recall that the original motivation for scramble number was to provide a lower bound on divisorial gonality.  Since \(\sn(G)\leq \scw(G)\), a natural question to ask is the following.

\begin{question}
Do we have \(\scw(G)\leq \gon(G)\) for all graphs \(G\)?
\end{question}

A ``yes'' answer would mean that screewidth forms a strictly stronger lower bound on gonality than scramble number; that is, we would have \[\sn(G)\leq \scw(G)\leq \gon(G)\] for all \(G\).  Assuming that this holds, a natural question is then what triples of integers \(\ell\leq m\leq n\) could be the scramble number, screewidth, and the gonality of a graph.  Although we do not know even the possible pairs of scramble number and screewidth, we can provide a partial answer below. Note that neither the statement of our proposition, nor its proof, assumes that \(\scw(G)\leq \gon(G)\) for all \(G\).

\begin{proposition}
\begin{itemize}
    \item[(i)] For any \(n\geq 1\), there exists a graph \(G\) with \(\sn(G)= \scw(G)= \gon(G)=n\).
    \item[(ii)] There exists a graph \(G\) with \(\sn(G)< \scw(G)= \gon(G)\)
    \item[(iii)] For any \(m\geq  n> 1\), there exists a graph \(G\) with \(\sn(G)= \scw(G)=n\) and \(\gon(G)=m\)
    \item[(iv)] For any \(n\geq 4\), there exists a graph \(G\) with \(\sn(G)=n\) and \(\sn(G)< \scw(G)< \gon(G)\).
\end{itemize}
\end{proposition}

\begin{proof}
For (i) we can use \(G=K_{n+1}\), which has all invariants equal to \(n\).  For (ii), let \(G\) be the graph in Figure \ref{fig:scw_neq_sn}, which we have already seen has scramble number \(2\) and screewidth \(3\).  Placing \(3\) chips on any single vertex gives a positive rank divisor; and due to the collections of triple parallel edges, \(2\) chips cannot successfully cover all vertices.  Thus this graph has gonality \(3\).

For (iii), we construct a graph \(G\) as follows.  
Let \(G_m\) denote the multipath on \(2m-2\) vertices obtained from \(P_{2m-2}\) by doubling every edge. Note that this graph has first Betti number \(g=2m-3\).  By \cite[Example 4.9]{scramble}, \(\sn(G_m)=2\), and in fact \(\scw(G_m)=2\), as obtained by constructing a path tree-cut decomposition where every vertex is in its own bag.  Now, let \(G_m'\) be the graph obtained from \(G_m\) by subdividing one of each pair of edges \(m-2\) times.  Since scramble number and screewidth are preserved under subdivision, \(\sn(G_m')=\scw(G_m')=2\). We now argue that the gonality of \(G_m'\) is \(m\).  First, the divisor \(m\cdot v\), with \(v\) the leftmost vertex, has positive rank:  by firing \(v\), then \(v\) and the first vertex clockwise from it, and so on, we may move all \(v\) chips to the vertex shared by the first two cycles.  Continuing in this fashion allows us to move chips anywhere in the graph, so \(m\cdot v\) has positive rank, and thus \(\gon(G_m')\leq m\). For a lower bound, we consider the metric version \(\Gamma\) of the graph \(G_m'\).   We may apply \cite[Theorem 1.1]{pflueger_special_divisors} to show that  \(\Gamma\) is Brill-Noether general, meaning that \(\gon(\Gamma)=\left\lfloor \frac{g+3}{2}\right\rfloor=m\).  Since the gonality of a finite graph is at least the gonality of the metric version of that graph \cite[Theorem 1.3]{discrete_metric_different}, we have \(\gon(G_m')\geq m\), so \(\gon(G_m')=m\).

Now Let \(H_n\) denote the multipath  on \(n\) vertices obtained from \(P_n\) by replacing every edge by \(n\) parallel edges.  Considering the scramble whose eggs are the vertices, and the natural tree-cut decomposition with a path of bags of size \(1\), we find \(\sn(H_n)=\scw(H_n)=n\).  Moreover, \(\gon(H_n)=n\): placing \(n\) chips on any one vertex yields a positive rank divisor \cite[Remark 4.9]{tropical_brill}.

\begin{figure}[hbt]
    \centering
    \includegraphics{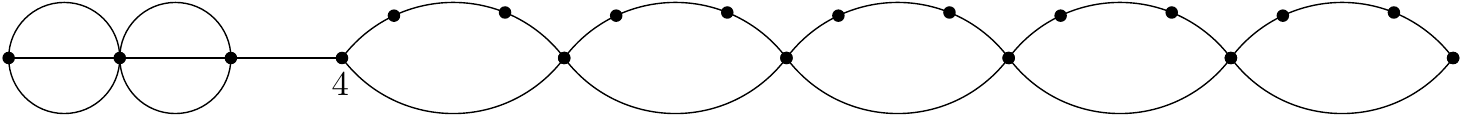}
    \caption{A graph with \(\sn(G)=\scw(G)=3\) and \(\gon(G)=4\).}
    \label{figure:graph_snscw3_gon4}
\end{figure}
Finally, we construct \(G\) by connecting \(G_m'\) and \(H_n\) with a bridge at the leftmost vertex of \(G_m'\) and the rightmost vertex of \(H_m\).  We illustrate \(G\) for \(n=3\) and \(m=4\) in Figure \ref{figure:graph_snscw3_gon4}.  Both its scramble number and its screewidth will equal the maximum of the corresponding invariant for the two graphs, which is \(\max\{2,n\}=n\) since \(n>1\).  The gonality of \(G\) must be at least the maximum gonality of its two connected components, namely \(\max\{m,n\}=m\); and placing \(m\) chips on the leftmost vertex of its \(G_m'\) subgraph does indeed have positive rank, since the chips can move freely across the bridge to \(H_n\).  Thus we have \(\sn(G)=\scw(G)=n\) and \(\gon(G)=m\) as desired.

Finally, for (iv) we may use the graph from the proof of Theorem \ref{theorem:scw_quadratic_sn}, for any choice of \(n\geq 4\).   Recall that this graph satisfied \(\sn(G)=n\) and \(\scw(G)=\left\lceil\frac{1}{4}n(n+2)\right\rceil-1\).  We claim that \(\gon(G)=n^2-1\).  There does exist  positive rank divisor of that degree, namely one that places a chip on every vertex except for one in the central copy of \(K_n\).  To see that we cannot win with fewer chips, we first note that due to the large number of parallel edges, it is impossible to chip-fire any vertex of a bulb without firing all other vertices of that bulb without introducing debt.  This means that any positive rank divisor must place a chip on each vertex outside of the central \(K_n\), since none of its neighbors can fire without the vertex also firing.  Moreover, given the structure of the graph, this means that we may treat each bulb as a single vertex, that either fires or does not.  Thus the minimum number of chips necessary in the central \(K_n\) is  equal to the gonality of \(K_n\), namely \(n-1\).  Taken together, we find that any positive rank divisor must have at least \(n(n-1)+(n-1)=n^2-1\), so \(\gon(G)=n^2-1\).  Thus we have a family of graphs with distinct scramble number, screewidth, and gonality.
\end{proof}

We remark that the construction for (iii) can be extended to construct a graph with \(\tw(G)=p\), \(\sn(G)=\scw(G)=n\) and \(\gon(G)=m\) for any \(1<p\leq n\leq m\) by attaching a complete graph \(K_p\) via a bridge to the rightmost vertex of \(G_m'\).  Thus no function of treewidth can be used to bound scramble number or screewidth.

We now prove that with an additional assumption on the divisors that achieve gonality on \(G\), we do have that \(\scw(G)\leq \gon(G)\).  If \(D\) is a positive rank divisor on \(G\) such that for every pair of effective divisors \(D',D''\) equivalent to \(D\) we have \(\supp(D')\cap \supp(D'')=\emptyset\), we say that \(D\) \emph{partitions} \(V(G)\).

\begin{theorem}\label{theorem:conditional_scw_gon}
Let \(G\) be a graph, and suppose that \(G\) has a positive rank divisor \(D\) that partitions \(V(G)\).  Then \(\scw(G)\leq \deg(D)\).  In particular, if \(\deg(D)=\gon(G)\), then \(\scw(G)\leq \gon(G)\).
\end{theorem}

\begin{proof}
It suffices to construct a tree-cut decomposition of width at most \(\deg(D)\).  Let \(\mathcal{X}=\{\supp(D')\,|\,D'\sim D, D'\geq 0\}\), and let \(T\) be the graph with a vertex \(b_{D'}\) for every \(D'\in |D|\), with an edge connecting \(b_{D'}\) and \(b_{D''}\) if and only if there is a subset-firing move that transforms \(D'\) into \(D''\).  We claim that \((T,\mathcal{X})\) is a tree-cut decomposition of \(G\).  See Figure \ref{figure:divisor_to_decomposition} for an example of this construction.

\begin{figure}[hbt]
    \centering
    \includegraphics[scale=0.7]{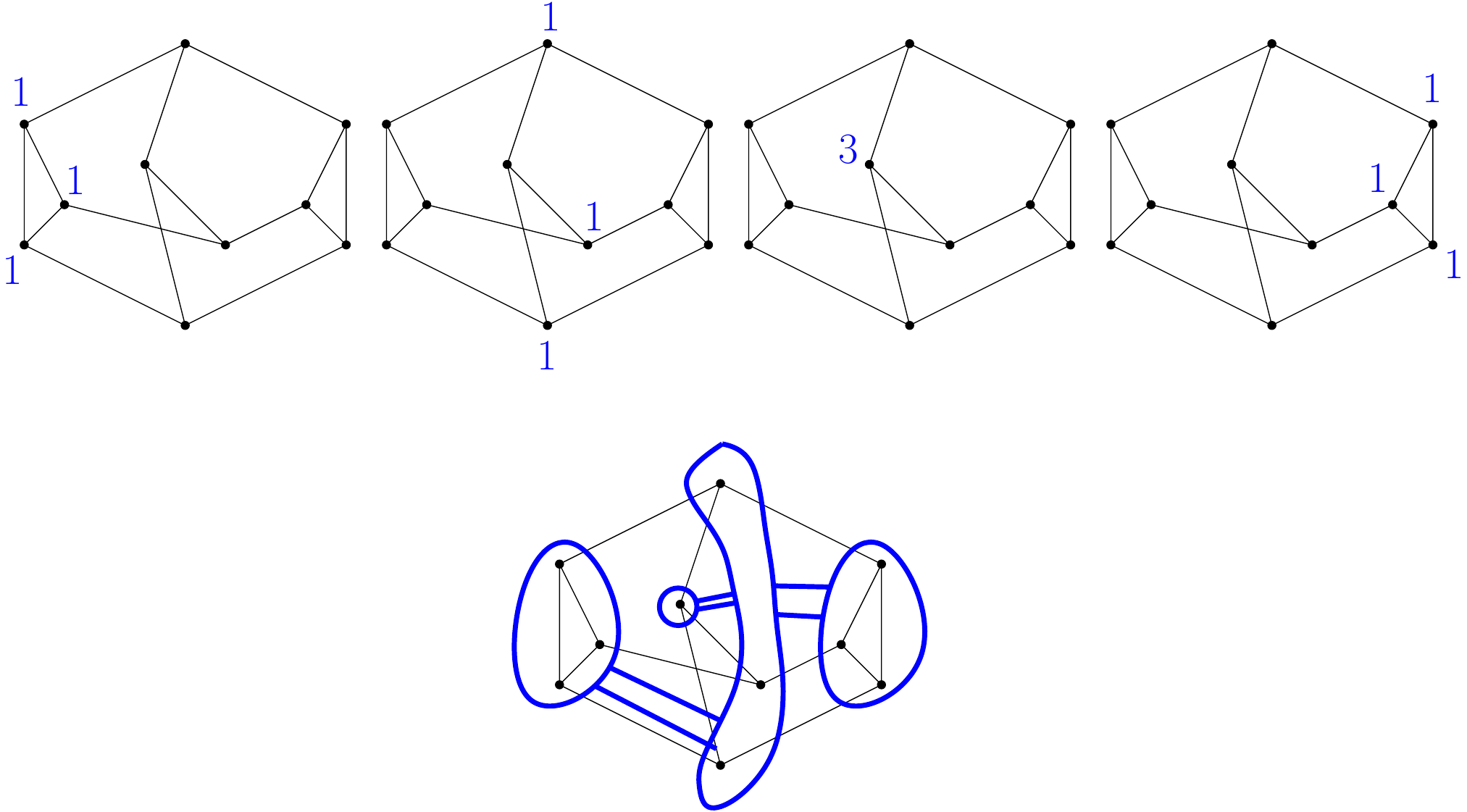}
    \caption{A graph with a class of effective divisors of positive rank with disjoint support, and the corresponding tree-cut decomposition.}
    \label{figure:divisor_to_decomposition}
\end{figure}

By assumption the elements of \(\mathcal{X}\) are disjoint, and since \(D\) has positive rank every vertex is contained in some element of \(\mathcal{X}\); thus \(\mathcal{X}\) is indeed a partition of \(V(G)\).  To show that \(T\) is a tree, we first note that it is connected:  \(b_{D'}\) and \(b_{D''}\) are connected by a path corresponding to the level-set decomposition of a set of firing moves transforming \(D'\) to \(D''\).  Now we argue \(T\) contains no cycles.  Consider two adjacent nodes \(b_1\) and \(b_2\), corresponding to divisors \(D_1\) and \(D_2\).  Let \(U\subset V(G)\) be the subset of vertices that, when fired, transforms \(D_1\) into \(D_2\).  We note that \(E(\supp(D_1),\supp(D_2))\) forms an edge-cut for \(G\).  In particular, deleting those edges splits \(G\) into \(G[U]\) and \(G[U^C]\).

Now suppose for the sake of contradiction that \(T\) contains a cycle, say \(b_1-b_2-\cdots-b_k-b_1\) where \(k\geq 2\) and all \(b_i\) are distinct. This means we have a collection of distinct effective divisors \(D_1,D_2,\ldots,D_k\) equivalent to \(D\) and a collection of subsets \(U_1,U_2,\ldots,U_k\) such that firing \(U_i\) transforms \(D_i\) into \(D_{i+1}\), where we take \(D_{k+1}=D_1\). Since \(E(\supp(D_1),\supp(D_2))\) forms an edge-cut of the graph, in which all chips of \(D_1\) are in one component and all chips of \(D_2\) are in another, we note that in order for chips to pass back to \(\supp(D_1)\), they must first pass back through \(\supp(D_2)\).  But then either \(D_k=D_2\), contradicting the distinctness of our divisors; or \(D_k\) and \(D_2\) are distinct with overlapping support, contradicting our starting assumption.  Thus \(T\) cannot contain a cycle, and as a connected graph must be a tree. This means that \(\mathcal{T}\) is a tree-cut decomposition of \(G\).  

We now determine \(w(\mathcal{T})\).  First we note that if \(l\) is a link with end-nodes \(b_1,b_2\) and \(e\in \adh(l)\), then the endpoints of \(e\) are in \(B_{b_1}\) and \(B_{b_2}\); this is because the edges connecting the support of two adjacent divisors form an edge cut for the graph.  A similar argument shows that the adhesion of any node is empty.  Every link in the graph has the same adhesion, namely \(\deg(D)\), since all \(\deg(D)\) chips are transferred by a subset-firing move.  Since \(|\supp(D')|\leq \deg(D)\) for every effective divisor \(D'\) equivalent to \(D\), we have that \(|X_b|\leq \deg(D)\) for all \(b\in V(T)\).  Thus \(w(\mathcal{T})=\deg(D)\), and we conclude that \(\scw(G)\leq \deg(D)\).
\end{proof}

\begin{corollary}\label{corollary:k_edge_connected}
If \(G\) is a \(k\)-edge-connected graph of gonality \(k\), then \(\sn(G)=\scw(G)= \gon(G)\).
\end{corollary}

\begin{proof} By \cite[Lemma 10.24]{sandpiles}, any positive rank divisor of degree \(k\) on a \(k\)-edge-connected graph partitions the \(V(G)\).  It follows immediately from Theorem \ref{theorem:conditional_scw_gon} that \(\scw(G)\leq \gon(G)=k\).

By \cite[Lemma 2.6]{echavarria2021scramble}, \(\sn(G)\geq \min\{|V(G)|,\lambda(G)\}\) for any graph \(G\).  Since \(|V(G)|\geq \gon(G)=k\) and \(\lambda(G)\geq k\), we have \(\sn(G)\geq k\).  We now have \(k\leq \sn(G)\leq \scw(G)\leq \gon(G)=k\), letting us conclude that \(\sn(G)=\scw(G)=\gon(G)=k\).
\end{proof}

We can use these results to compute the screewidth of many graphs.

\begin{proposition}\label{prop:scw_families}
\begin{itemize}
    \item[(i)] If \(C_n\) is a cycle graph, then \(\scw(C_n)=2\)
    \item[(ii)] For a \(k\)-partite graph \(K_{n_1,n_2,\ldots,n_k}\), we have \(\scw(K_{n_1,n_2,\ldots,n_k})=\sum_{i}n_i-\max_i\{n_i\}\).
    \item[(iii)] If \(G_{m,n}\) denotes the \(m\times n\) grid graph, then \(\scw(G_{m,n})=\min\{m,n\}\).
    \item[(iv)] If \(Y_{m,n}\) denotes the \(m\times n\) stacked prism graph, then \(\scw(Y_{m,n})=\min\{m,2n\}\).
    \item[(v)] If \(T_{m,n}\) denotes the \(m\times n\) toroidal grid graph, then \(\scw(T_{m,n})=\min\{2m,2n\}\).
    \item[(vi)] If \(Q_n\) denotes the \(n\)-dimensional hybercube graph, then \(\scw(Q_n)\leq2^{n-1}\), with equality for \(n\leq 5\).
\end{itemize}
\end{proposition}

\begin{proof} A cycle graph is a \(2\)-edge-connected graph of gonality \(2\), and \(K_{n_1,n_2,\ldots,n_k}\) is an \(\ell\)-edge-connected graph of gonality \(\ell\), where \(\ell=\sum_{i}n_i-\max_i\{n_i\}\) \cite[\S 4]{debruyn2014treewidth}. Thus Corollary \ref{corollary:k_edge_connected} immediately proves claims (i) and (ii).

The graph \(G_{m,n}\) has divisors of positive rank and degrees \(m\) and \(n\) that each partition \(V(G_{m,n})\), namely a divisor placing a chip on every vertex in a single column, and a divisor placing a chip on every vertex in a single row.  By Theorem \ref{theorem:conditional_scw_gon} we have \(\scw(G_{m,n})\leq \min\{m,n\}\); and since \(\sn(G_{m,n})=\min\{m,n\}\) \cite[Proposition 5.2]{scramble}, we have \(\scw(G_{m,n})= \min\{m,n\}\).  A nearly  identical argument holds for \(T_{m,n}\), except with placing a chip on every vertex in a pair of rows or a pair of columns; and a hybrid argument holds for \(Y_{m,n}\).  See \cite[Propositions 5.3 and 5.4]{scramble} for the relevant results on scramble number.

For \(Q_n\), we again have a positive rank divisor \(D\) of degree \(2^{n-1}\) partitioning \(V(Q_n)\):  viewing \(Q_n=Q_{n-1}\square K_2\), place a single chip on every vertex in one copy of \(Q_{n-1}\).  For \(n\leq 5\), we have \(\sn(Q_{n})=2^{n-1}\) (see \cite{scramble} for \(n\leq 4\), and \cite{uniform_scrambles} for \(n=5\)), giving \(\scw(Q_{n})=2^{n-1}\) for small \(n\).
\end{proof}

In general, a graph may not have a divisor achieving gonality that partitions the vertices of \(V(G)\).  We illustrate with an example another possible approach to handle this situation.  Intuitively, we use a positive rank divisor to probe all vertices of the graph, gradually building up a tree-cut decomposition as we go.

\begin{example}
Consider the Sierpinski graph illustrated in Figure \ref{figure:sierpinski_1}, with a divisor \(D\) of degree \(6\) and positive rank.

\begin{figure}[hbt]
    \centering
    \includegraphics{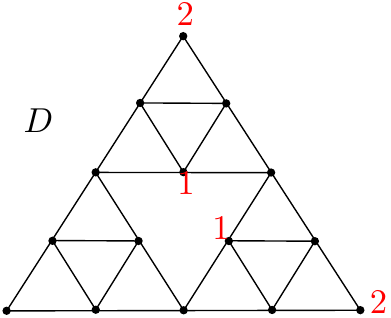}
    \caption{The Sierpinski graph with a positive rank divisor of degree 6.}
    \label{figure:sierpinski_1}
\end{figure}

We choose \(v\in V(G)- \supp(D)\).  Since \(D\) has positive rank, there exists at least one effective divisor \(D'\) with \(D\sim D'\) and \(v\in \supp(D')\); among all such divisors, choose \(D'\) to be the one that requires the minimum number of subset-firing moves to reach from \(D\). 
We can then consider the level-set decomposition \(A_0\subseteq A_1\subseteq\ldots \subseteq A_n\) to get from \(D\) to \(D'\). Let \(D_i\) be the configuration achieved after firing \(A_0,\ldots,A_{i-1}\), so \(D=D_0\) and \(D'=D_n\). Figure \ref{figure:sierpinski_2} presents the appropriate sequence of subset-firing moves.

\begin{figure}[hbt]
    \centering
    \includegraphics[scale=0.8]{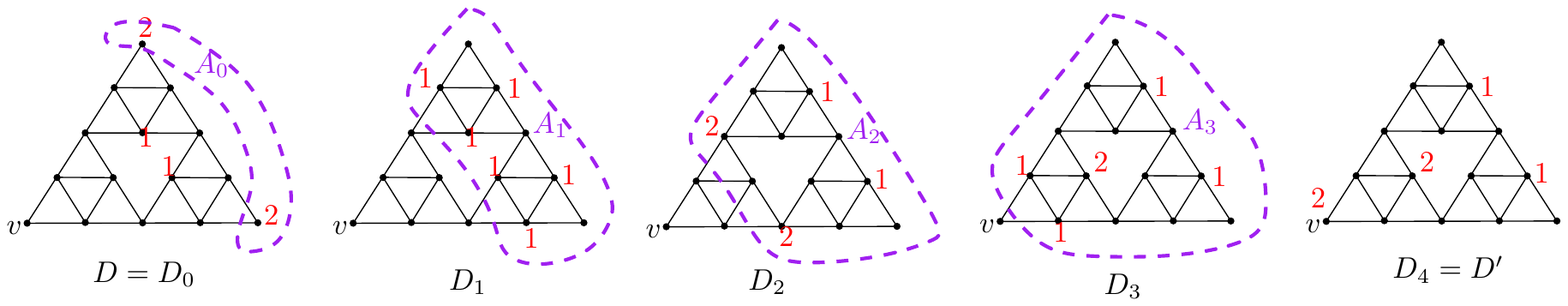}
    \caption{A sequence of subset-firing moves to place a chip on the vertex \(v\).}
    \label{figure:sierpinski_2}
\end{figure}

We start to build a tree-cut decomposition as follows.  Start with a trivial tree-cut decomposition, and turn it into a path of \(n+1\) nodes \(b_0,\ldots,b_n\) with a link between each successive node such that \(X_{b_0}=A_0\), \(X_j=A_j-A_{j-1}\) for \(1\leq j\leq n-1\), and \(X_n=V(G)-A_n\).
 This is illustrated in Figure \ref{figure:sierpinski_3}.
 
 \begin{figure}[hbt]
    \centering
    \includegraphics[scale=0.85]{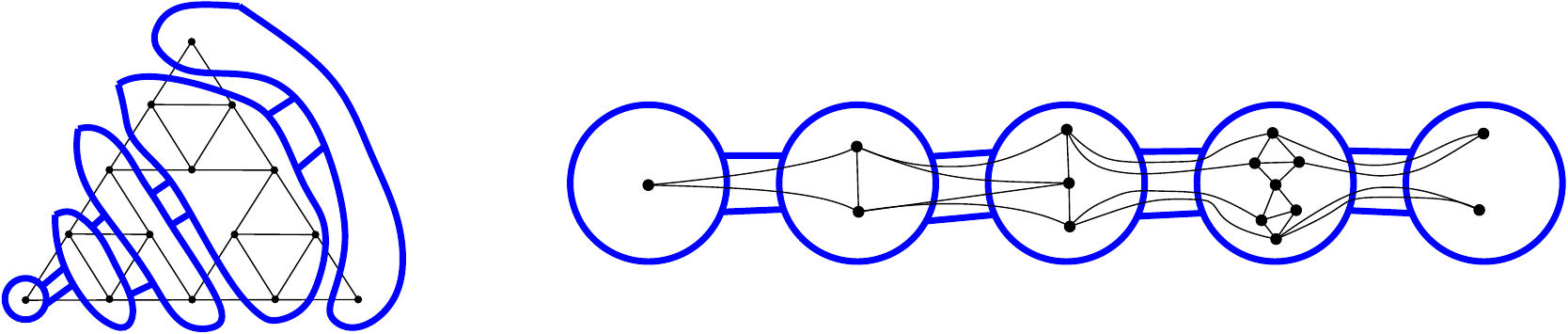}
    \caption{A tree-cut decomposition induced by the subset-firing moves.}
    \label{figure:sierpinski_3}
\end{figure}

\begin{figure}[hbt]
    \centering
    \includegraphics{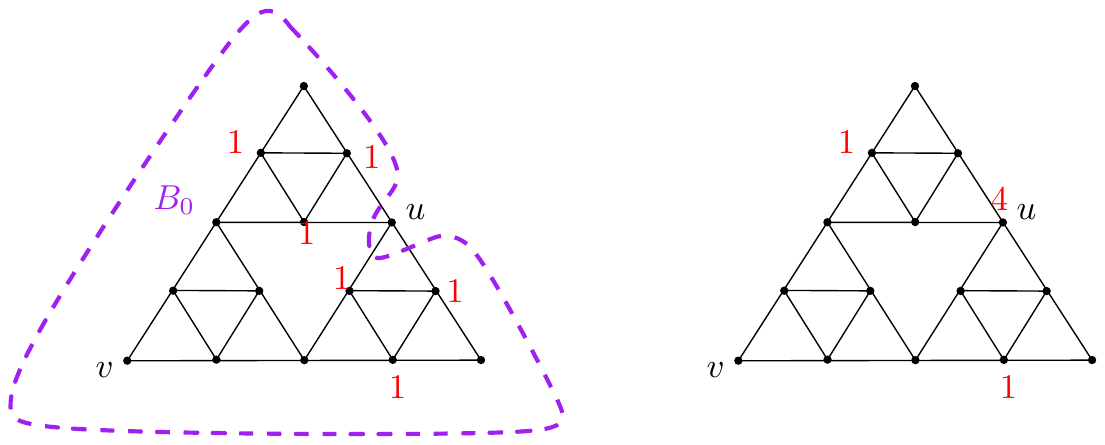}
    \caption{A subset-firing move to place a chip on the last vertex.}
    \label{figure:sierpinski_4}
\end{figure}
 
Note that we can pick a (unique, in our example) vertex \(u\) that has not yet been in the support of any of divisors.  We have that \(u\in X_k\) for some \(k\); consider \(D_k\).  Of all the effective divisors equivalent to \(D\) with \(u\) in their support, let \(D''\) be the one that can be reached from \(D_k\) in the minimum number of subset-firing moves.  We then consider the level set decomposition \(B_0\subseteq B_1\subseteq\ldots\subseteq B_m\) of moving from \(D_k\) to \(D_k''\).  For each \(B_i\), we have a corresponding divisor \(D_i'\), namely the divisor obtained by starting with \(D_k\) and firing \(B_0,B_1,\ldots,B_{i-1}\).  This is illustrated in Figure \ref{figure:sierpinski_4} for the only possible choice of \(u\).

Now modify our present tree-cut decomposition by turning \(X_k\) into a path on \(m+1\) bags rooted at \(X_k\), where \(X_k'=X_k\cap B_0\), \(X_j'=X_k\cap (B_j-B_{j-1})\) for \(1\leq j\leq m-1\), and \(X_m'=V(G)-B_m\).  This is illustrated in Figure \ref{figure:sierpinski_5}.

\begin{figure}[hbt]
    \centering
    \includegraphics[scale=0.85]{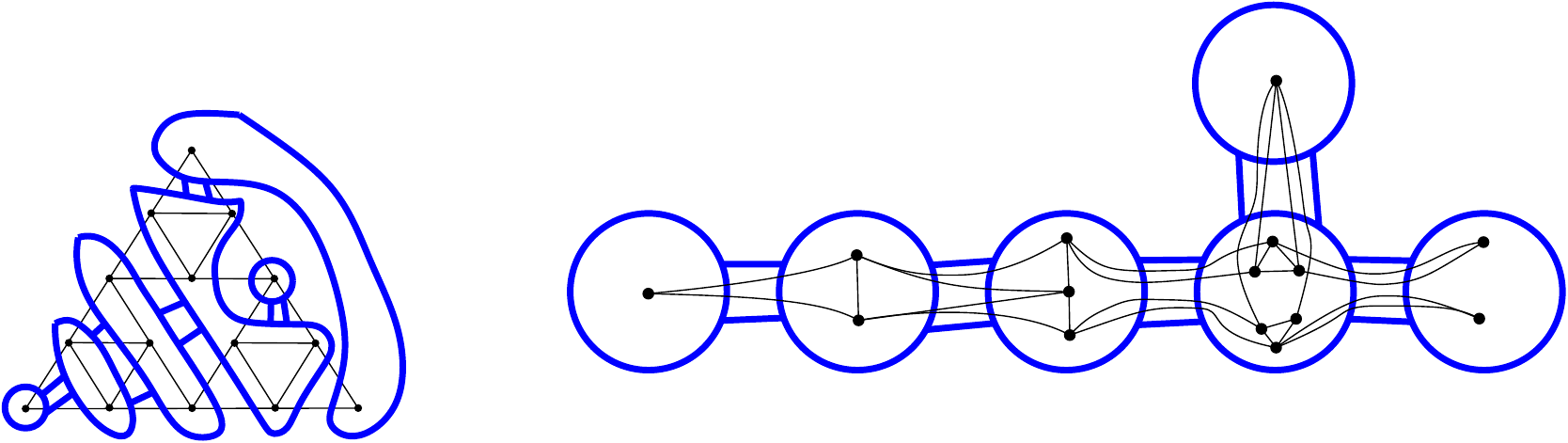}
    \caption{An extension of the tree-cut decomposition for the Sierpinski graph, with width equal to the degree of the starting divisor.}
    \label{figure:sierpinski_5}
\end{figure}

At this point we have reached a tree-cut decomposition of width \(6=\deg(D)\), and we finish.
\end{example}

\begin{figure}[hbt]
    \centering
    \includegraphics[scale=0.7]{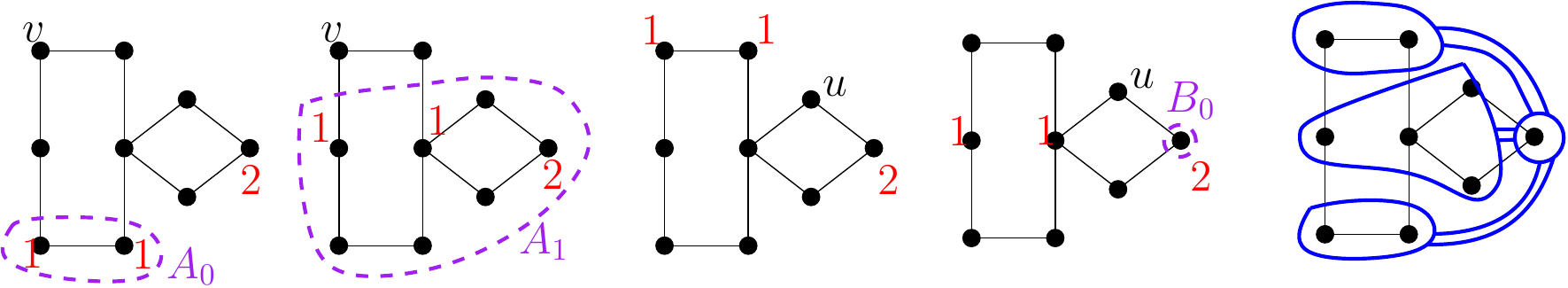}
    \caption{A sequence of subset-firing moves yielding a tree-cut decomposition of high width.}
    \label{figure:counterexample_remains}
\end{figure}

Unfortunately this strategy does not work in general, at least for arbitrary choices of subset-firing moves.  Consider for instance the graph in Figure \ref{figure:counterexample_remains}, with a positive rank divisor of degree \(4\) and a possible way to follow our strategy. (We remark that the gonality of this graph is \(2\) rather than \(4\), but it remains an illustrative example.) The resulting tree-cut decomposition has a width of \(6\) due to an edge adhesion, even though the divisor has degree \(4\).

\begin{figure}[hbt]
    \centering
    \includegraphics[scale=0.75]{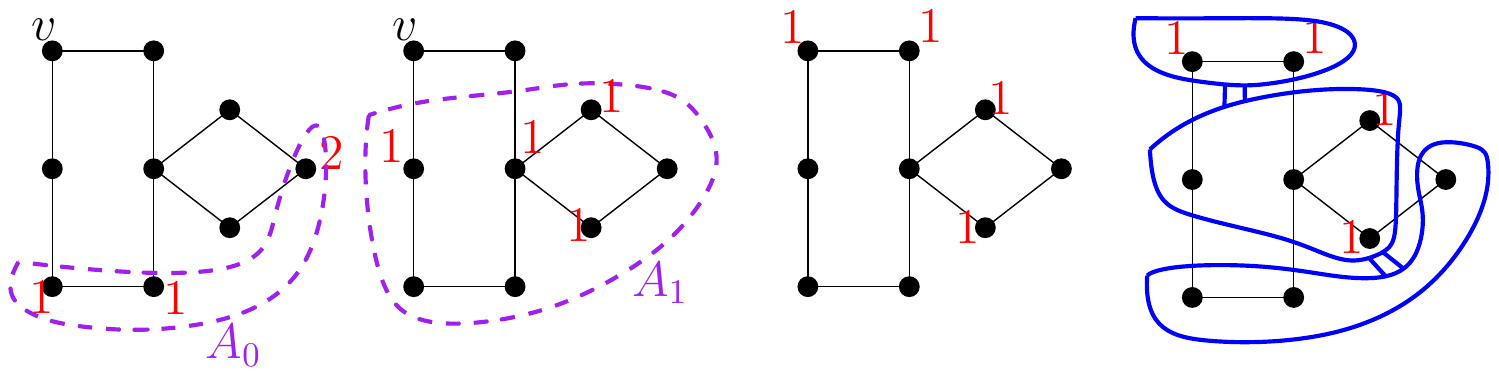}
    \caption{A sequence of subset-firing moves yielding a tree-cut decomposition of low width.}
    \label{figure:counterexample_remains_but_dhars}
\end{figure}

However, if we follow a more natural choice of subset-firing moves, namely the largest possible subsets that move chips towards the target vertex without introducing debt (found e.g. using Dhar's burning algorithm \cite{dhar}), we can obtain a tree-cut decomposition of width \(4\); see Figure \ref{figure:counterexample_remains_but_dhars}.  An interesting direction for future work would be to determine whether this strategy can always be used to give a tree-cut decomposition of width at most \(\deg(D)\), thereby proving that \(\scw(G)\leq \gon(G)\) for all graphs \(G\).

We can find large families of graphs for which \(\sn(G)=\scw(G)=\gon(G)\) by turning to product graphs.  

\begin{proposition}\label{prop:scw_product}
The screewidth of the Cartesian product of two graphs \(G\) and \(H\) satisfies the following inequality:
\[\scw(G\square H)\leq \min\{|V(H)|\scw(G),|V(G)|\scw(H)\}.\]
\end{proposition}

\begin{proof}
Let \(\mathcal{T}=(T,\mathcal{X})\) be tree-cut decomposition of \(G\) of width \(\scw(G)\). For every bag \(B\subset V(G)\) in  \(\mathcal{X}\), let \(B'\subset V(G\square H)\) be the set \(\{(u,v)\,|\, u\in B\}\), and define \(\mathcal{X}'=\{B'\,|\, B\in \mathcal{X}\}\).  Consider the tree-cut decomposition  \(\mathcal{T}'=(T,\mathcal{X}')\) on \(G\square H\).  Every bag size has increased by a factor of \(|V(H)|\), as has every link adhesion and node adhesion.  It follows that \(w(\mathcal{T}')= |V(H)|w(\mathcal{T})=|V(H)|\scw(G)\), so \(\scw(G\square H)\leq |V(H)|\scw(G)\).  By symmetry, we have
\[\scw(G\square H)\leq \min\{|V(H)|\scw(G),|V(G)|\scw(H)\}.\qedhere\]
\end{proof}

We remark that this upper bound may be a strict inequality.  In Figure \ref{figure:scw_product_sctrict} we illustrate a graph \(G\) on three vertices, the product \(G\square G\), a scramble on \(G\square G\), and a tree-cut decomposition on \(G\square G\).  We remark that \(\scw(G)=2\), that the order of the scramble is \(4\), and that the width of the tree-cut decomposition is \(4\); thus, even though \(|V(G)|\scw(G)=3\cdot 2=6\), we have \(\scw(G\square G)=4<6\).  This same graph was used in \cite[Figure 1]{aidun2019gonality} to prove that the gonality of a product graph could be strictly smaller than a similar upper bound; it turns out the gonality of this graph is equal to~\(5\).

\begin{figure}[hbt]
    \centering
    \includegraphics{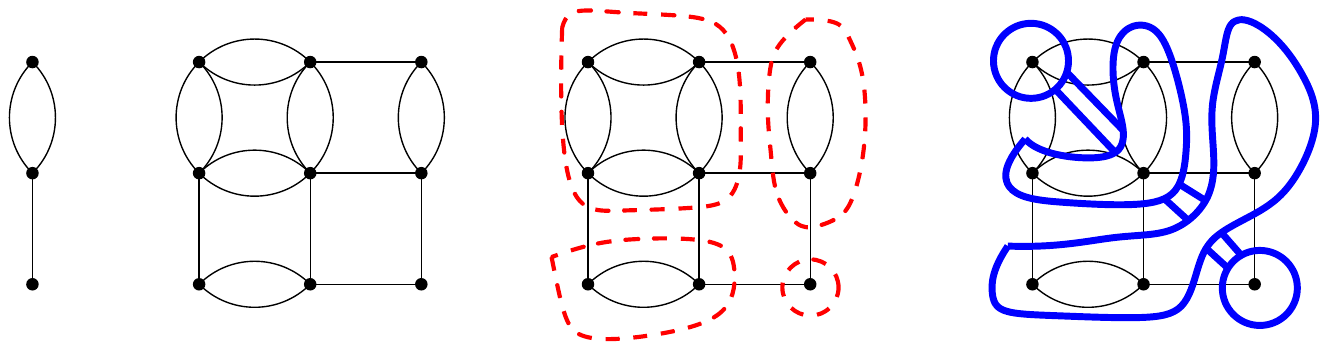}
    \caption{A graph \(G\), the product \(G\square G\), a scramble of order \(4\), and a tree-cut decomposition of width \(4\).}
    \label{figure:scw_product_sctrict}
\end{figure}

There are many families of product graphs for which the upper bound on screewidth equals a lower bound on scramble number, giving us an exact formula for screewidth.

\begin{corollary} We have the following screewidth formulas.  In every case, the scramble number of the graph is given by the same formula; and for cases (ii), (iv), (v), (vii), and (viii), so is the gonality of the graph.

\begin{itemize}
    \item[(i)] \(\scw(T\square H)=\min\{|V(H)|,\scw(H)|V(G)|\}\), where \(T\) is a tree and \(H\) is a a graph with \(\lambda(H)=\scw(H)\).
    \item[(ii)]\(\scw(G\square K_n\square T)=n|V(G)|\), where \(T\) is a tree and \(G\) is any graph with \(|V(G)|\leq |V(T)|\).
    \item[(iii)] \(\scw(G\square H)=2|V(G)|\), where \(\lambda(H)=\scw(H)=2\) and \(|V(G)|\leq |V(H)|/2\).
    \item[(iv)]\(\scw(G\square K_{m,n})=|V(G)|\min\{m,n\}\), where \(|V(G)|\leq (m+n)/m\).
    \item[(v)] \(\scw(C_m\square K_n)=\min\{2n,m(n-1)\}\) for \(m\geq 4\) and \(n\geq 2\).
    \item[(vi)] \(\scw(G\square H)=2\min\{|V(G)|,|V(H)|\}\), where \(\kappa(G)=\kappa(H)=\scw(G)=\scw(H)=2\).
    \item[(vii)] \(\scw(C_\ell\square C_m\square C_n)=2\ell m\), where \(\ell,m\geq 3\) and \(\frac{2}{3}\ell m\leq n\).
    \item[(viii)] \(\scw(K_k\square T\square K_\ell)=k\ell\), where \(T\) is a tree and \(k<\ell\leq k(|V(T)|-2)+4\).
\end{itemize}

\end{corollary}

\begin{proof}
For every claim, the upper bound on screewidth is provided by Proposition \ref{prop:scw_product}.  The lower bounds are provided by matching formulas for scramble number proved in \cite[\S 5]{echavarria2021scramble}, which also gives the relevant formulas for gonality \cite[Table 1]{echavarria2021scramble}.
\end{proof}

We close now with several families of graphs for which screewidth is unknown.
\begin{itemize}
    \item[(a)] Rook's graphs \(R_{m,n}=K_m\square K_n\). It is known that \(\gon(R_{m,n})=\min\{m(n-1),n(m-1)\}\) \cite[Theorem 1.1]{rooks_gonality}, and also that \(R_{m,n}\) has a divisor \(D\) of positive rank achieving gonality and partitioning \(V(R_{m,n})\) (in particular, place a single chip on every vertex except those in one column). Thus we have \(\scw(R_{m,n})\leq \min\{m(n-1),n(m-1)\}\).  When \(\min\{m,n\}\in \{2,3\}\), or when \(m\geq (n-1)(m-2)\) this formula is equal to scramble number \cite[Theorem 1.2 and Lemmas 5.2 and 5.3]{rooks_gonality}, implying that \(\scw(R_{m,n})= \min\{m(n-1),n(m-1)\}\) in these cases.  However, once \(\min\{m,n\}\geq 4\), we have that \(\sn(R_{m,n})<\gon(R_{m,n})\), meaning there is in general uncertainty in the behavior of \(\scw(R_{m,n})\).  We can compute \(\scw(R_{4,4})\) with the tree-cut decomposition in Figure \ref{figure:r44_scw}, in which the edges connecting vertices have been omitted.  The width of this tree-cut decomposition is \(11\), and \(\sn(R_{4,4})=11\) \cite[Example 5.4]{rooks_gonality}, so \(\scw(R_{4,4})=11<12=\gon(R_{4,4})\).
    
    \begin{figure}
        \centering
        \includegraphics{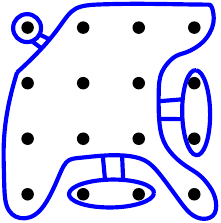}
        \caption{A tree-cut decomposition of width \(11\) on the \(4\times 4\) Rook's graph.}
        \label{figure:r44_scw}
    \end{figure}
    
    \item[(b)] Hypercube graphs \(Q_n\).  Although we showed \(\scw(Q_n)=2^{n-1}\) for \(n\leq 5\), this case is open for all other \(n\).  Indeed, neither scramble number nor gonality are known for \(Q_6\), although it is known that \(\sn(Q_6)<32\) \cite{uniform_scrambles} and it is conjectured that \(\gon(Q_6)=32\) \cite{debruyn2014treewidth}.
\end{itemize}

\appendix

\bibliographystyle{plain}

\end{document}